\documentclass[letterpaper,11pt]{amsart}
\usepackage{amscd}
\usepackage{amssymb}
\usepackage{amsthm}
\usepackage{enumerate}
\usepackage[normalem]{ulem}
\usepackage{mathtools}
\usepackage[usenames,dvipsnames]{xcolor}
\usepackage{stmaryrd}
\usepackage[mathscr]{eucal}
\usepackage{cite}
\usepackage{upgreek}
\usepackage[bookmarks=false]{hyperref}
\usepackage[T2A]{fontenc}
\usepackage[utf8]{inputenc}
\usepackage{mathrsfs}
\usepackage[margin=1.2in]{geometry}

\newtheorem{introthm}{Theorem}
\newtheorem{introcor}[introthm]{Corollary}

\newtheorem{thm}{Theorem}[section]
\newtheorem{lem}[thm]{Lemma}
\newtheorem{prop}[thm]{Proposition}
\newtheorem{cor}[thm]{Corollary}

\theoremstyle{definition}
\newtheorem{defn}[thm]{Definition}
\newtheorem{notn}[thm]{Notation}

\theoremstyle{remark}
\newtheorem{rem}[thm]{Remark}

\numberwithin{equation}{section}

\newcommand{\bC}{{\mathbb C}}

\newcommand{\bN}{{\mathbb N}}

\newcommand{\bR}{{\mathbb R}}

\newcommand{\bZ}{{\mathbb Z}}

\newcommand{\cA}{{\mathcal A}}
\newcommand{\cB}{{\mathcal B}}
\newcommand{\cC}{{\mathcal C}}
\newcommand{\cD}{{\mathcal D}}

\newcommand{\cL}{{\mathcal L}}
\newcommand{\cM}{{\mathcal M}}
\newcommand{\cN}{{\mathcal N}}
\newcommand{\cO}{{\mathcal O}}
\newcommand{\cP}{{\mathcal P}}
\newcommand{\cQ}{{\mathcal Q}}
\newcommand{\cR}{{\mathcal R}}
\newcommand{\cS}{{\mathcal S}}

\newcommand{\cU}{{\mathcal U}}
\newcommand{\cV}{{\mathcal V}}
\newcommand{\cW}{{\mathcal W}}

\DeclareMathOperator{\re}{Re}

\DeclareMathOperator{\tr}{tr}

\DeclareMathOperator{\Tr}{Tr}

\DeclareMathOperator{\ad}{ad}

\DeclareMathOperator{\orb}{orb}

\DeclareMathOperator{\diag}{diag}
\DeclareMathOperator{\Ent}{Ent}
\DeclareMathOperator{\tp}{tp}
\DeclareMathOperator{\Lip}{Lip}
\DeclareMathOperator{\Haar}{Haar}

\newcommand{\ip}[1]{\langle #1 \rangle}

\DeclarePairedDelimiter{\norm}{\lVert}{\rVert}

\begin{document}

	\title[Upgraded free independence phenomena for random unitaries]{Upgraded free independence phenomena \\ for random unitaries}

	\author{David Jekel}
	\address{\parbox{\linewidth}{Department of
			Mathematics and Statistics,
			York University \\
			4700 Keele Street, Toronto, ON M3J1P3, Canada}}
	
	\address[Current address]{\parbox{\linewidth}{Department of
			Mathematical Sciences,
			University of Copenhagen \\
			Universitetsparken 5, 2100 Copenhagen \O, Denmark}}
	\email{daj@math.ku.dk}
	\urladdr{http://davidjekel.com}

	\author{Srivatsav Kunnawalkam Elayavalli}
	\address{\parbox{\linewidth}{Department of Mathematics, University of California, San Diego, \\
			9500 Gilman Drive \# 0112, La Jolla, CA 92093, USA}}
	\email{srivatsav.kunnawalkam.elayavalli@vanderbilt.edu}
	\urladdr{https://sites.google.com/view/srivatsavke}
	
	\address[Current address]{\parbox{\linewidth}{Department of Mathematics, University of Maryland, College Park \\ William E. Kirwan Hall, 4176 Campus Dr, College Park, MD 20742, USA}} 
	\email{sriva@umd.edu}
	
	
	\begin{abstract}
		We study upgraded free independence phenomena for unitary elements $u_1$, $u_2$, \dots representing the large-$n$ limit of Haar random unitaries, showing that free independence extends to several larger algebras containing $u_j$ in the ultraproduct of matrices $\prod_{n \to \cU} M_n(\bC)$. Using a uniform asymptotic freeness argument and volumetric analysis, we  prove free independence of the Pinsker algebras $\cP_j$ containing $u_j$. The Pinsker algebra $\cP_j$ is the maximal subalgebra containing $u_j$ with vanishing $1$-bounded entropy \cite{Hayes2018}; $\cP_j$ in particular contains the relative commutant $\{u_j\}' \cap \prod_{n \to \cU} M_n(\bC)$, more generally any unitary that can be connected to $u_j$ by a sequence of commuting pairs of Haar unitaries, and any unitary $v$ such that $v\cP_j v^* \cap \cP_j$ is diffuse.  Through an embedding argument, we go back and deduce analogous free independence results for $\cM^{\cU}$ when $\cM$ is a free product of Connes embeddable tracial von Neumann algebras $\cM_i$, which thus yields (in the Connes-embeddable case) a generalization and a new proof of Houdayer--Ioana's results on free independence of approximate commutants \cite{houdayer2023asymptotic}.  It also yields a new proof of the general absorption results for Connes-embeddable free products obtained by the first author, Hayes, Nelson, and Sinclair \cite{freePinsker}.
	\end{abstract}
	
	\maketitle%

	\section{Introduction}
	
	\subsection{Main results}
	
	A fundamental result in random matrix theory is that for independent $n \times n$ Haar random unitaries $U_1^{(n)}$, $U_2^{(n)}$, \dots, the trace of any word in the $U_j^{(n)}$'s and their adjoints converges to trace of the corresponding word in the free group \cite{VoicAsyFree,Voiculescu1998}.  Thus, the von Neumann algebra of the free group $L(F_\infty)$ arises from the large $n$ limit of random matrices.  Recently, Houdayer and Ioana \cite{houdayer2023asymptotic} discovered that free independence of algebras $\cM_i$ in a free product $\cM = *_{i \in I} \cM_i$ can be upgraded to free independence of their commutants in the ultrapower of $\cM$.  Thus, for instance, if $b_1$, $b_2$, \dots are elements in $L(F_\infty)$ such that $b_j$ approximately commutes with the group generator $g_j$, then $b_1$, $b_2$, \dots must be approximately freely independent.  Thus, it is natural to ask whether an analogous statement holds for Haar random unitary matrices.  We will give an affirmative answer to this question, so that matrix ultraproducts exhibit similar behavior to free products in this regard, and in addition prove several generalizations.
	
	\begin{introthm}[Asymptotic freeness of approximate commutants] \label{thm: random matrix commutant}
		Let $U_1^{(n)}$, $U_2^{(n)}$, \dots be independent $n \times n$ Haar random unitary matrices.  Let $B_1^{(n)}$, \dots, $B_m^{(n)}$ be random matrices on the same probability space such that $\norm{B_j^{(n)}} \leq 1$ and $\lim_{n \to \infty} \norm{[U_j^{(n)},B_j^{(n)}]}_2 = 0$ almost surely for $j = 1, \dots, k$.  Then $B_1^{(n)}$, \dots, $B_k^{(n)}$ are almost surely asymptotically freely independent.
	\end{introthm}
	
	Here also each matrix $B_j^{(n)}$ can be replaced by a tuple; see Theorem \ref{thm: random matrix commutants 2}.  The challenge of Theorem \ref{thm: random matrix commutant} is that $B_j^{(n)}$ is allowed to depend arbitrarily on the $U_j^{(n)}$.  Thus, the proof requires a version of Voiculescu's asymptotic freeness \cite[Corollary 2.7]{Voiculescu1998} that applies \emph{uniformly} to all the possible values of $B_j^{(n)}$.  After considering the diagonalization of the $U_j^{(n)}$'s (see \S \ref{subsec: diagonalization}), we will show a uniform asymptotic freeness result for matrices $B_j^{(n)}$ that are asymptotically supported in $\epsilon n$-bands around the diagonal (Lemma \ref{lem: uniform asymptotic freeness}).   To guarantee each moment condition, we test it on a $\delta$-dense subset for some small $\delta$ by playing off the exponential concentration of measure for the Haar random unitary matrices against the small dimension of the $\epsilon$-bands compared to the ambient matrix space.
	
	Such asymptotic results can be conveniently formulated using ultraproducts of tracial von Neumann algebras (see \cite[Appendix A]{BO08} for background); intuitively, elements of the ultraproduct $\prod_{n \to \cU} \cM_n$ capture all possible limiting behaviors of elements $x_n$ from $\cM_n$, and thus allow asymptotic or approximate statements to be reformulated as exact statements.  Letting $\cU$ be a free ultrafilter on $\bN$, Houdayer--Ioana's result in the special case of $L(F_\infty)$ would say that the commutants $\{g_j\}' \cap L(F_\infty)^{\cU}$ are freely independent of each other.  Meanwhile, the ultraproduct version of Theorem \ref{thm: random matrix commutant} is the following.
	
	\begin{introthm}[Freeness of relative commutants] \label{thm: commutant}
		Let $\cU$ be a free ultrafilter on $\bN$, and let $\cQ = \prod_{n \to \cU} M_n(\bC)$ be the ultraproduct of matrix algebras.  Let $U_1^{(n)}$, $U_2^{(n)}$, \dots be independent $n \times n$ Haar random unitary matrices on a probability space $(\Omega,\mathcal{F},P)$.  For each $\omega \in \Omega$, let $u_j(\omega) = [U_j^{(n)}(\omega)]_{n \in \bN} \in \mathcal{Q}$ be the corresponding element of the matrix ultraproduct.  Then almost surely
		\[
		\{u_1(\omega)\}' \cap \mathcal{Q}, \quad \{u_2(\omega)\}' \cap \cQ, \quad \dots
		\]
		are freely independent.
	\end{introthm}
	
	This method of volumetric analysis on the space of matrices goes back at least to von Neumann \cite{vonNeumann1942}. Voiculescu formalized the exponential growth rate of volumes of the set of matrix microstate with certain moments through his free entropy $\chi$ \cite{VoiculescuFreeEntropy2}, which he used to show that the free group von Neumann algebra has no Cartan subalgebra \cite{VoiculescuFreeEntropy3} (see also \cite{GePrime}). Jung \cite{Jung2007} defined the related notion of strong $1$-boundedness, a condition of ``small microstate dimension'' which is independent of the choice of generators of the von Neumann algebra, and which Hayes later captured through the metric-entropy invariant $h$ \cite{Hayes2018}.  Volumetric analysis and high-dimensional concentration of measure in random matrix theory \cite{BAG1997,Ledoux2001} form a powerful combination with many applications to the structure of $L(F_\infty)$ \cite{freePinsker,HayesPT,hayes2023consequences}.  These techniques also relate closely to deep questions about the growth rates of approximate representations \cite{Pisier1,Pisier2,PisierSubExp}.
	
	Using the toolkit of $1$-bounded entropy and a small amount of model theory for von Neumann algebras \cite{FHS2013,FHS2014a,JekelCoveringEntropy}, we show that this upgrading of freeness phenomenon applies not only to the commutants of $u_j$ in $\cQ$, but to the much larger Pinsker algebra of $u_j$.  A \emph{Pinsker algebra} \cite[Definition after Theorem B]{freePinsker} in a von Neumann algebra $\cM$ is a maximal von Neumann algebra $\cP \subseteq \cM$ such that $h(\cP:\cM) = 0$ (this terminology is motivated by an analogous construction in ergodic theory).  Thanks to \cite[Lemma A.12]{Hayes2018}, every diffuse von Neumann subalgebra $\cA \subseteq \cM$ with $h(\cA:\cM) = 0$ is contained in a unique Pinsker algebra $\cP$.  Moreover, the general properties of $1$-bounded entropy (see e.g.\ \cite[\S 1.2]{freePinsker}, \cite[\S 2.3]{PropTS1B}) imply the following useful properties of the Pinsker algebra $\cP$.  If $\cB \supseteq \cA$ is amenable or has property Gamma, then $\cB$ must be contained in $\cP$.  If $u$ is unitary and $u \cP u^* \cap \cP$ is diffuse, then $u$ must be contained in $\cP$, and thus in particular $\cP$ contains many different weakened versions of normalizers as in \cite{IPP08,PetersonThom,FangGaoSmith,GalatanPopa}.  In fact, $L^2(\cM) \ominus L^2(\cP)$ is a coarse bimodule over $\cP$ \cite[Theorem 3.8]{Hayes2018}.
	
	\begin{introthm}[Freeness of Pinsker algebras] \label{thm: freeness of Pinsker}
		Let $u_j(\omega)$ be as in Theorem \ref{thm: commutant}.  Let $\cP_j$ be the Pinsker algebra of $u_j(\omega)$.  Then almost surely $\cP_1$, $\cP_2$, \dots are freely independent.
	\end{introthm}
	
	This immediately implies the following corollary, for instance.
	
	\begin{introcor}[Freeness of amenable algebras]
		Let $u_j(\omega)$ be as in Theorem \ref{thm: commutant}.  Then almost surely the following statement holds:  If $\cA_j$ is any amenable subalgebra containing $u_j(\omega)$, then $\cA_1$, $\cA_2$, \dots are freely independent.
	\end{introcor}
	
	Another consequence of Theorem \ref{thm: freeness of Pinsker} is freeness of the sequential commutation orbits of $u_j$ studied by \cite{patchellelayavalli2023sequential}; see also \cite{gao2024internal}, \cite{exoticCIKE}.  Recall a \emph{Haar unitary} in a tracial von Neumann algebra $\cM$ is any unitary element $u$ satisfying $\tr^{\cM}(u^m) = 0$ for all $m \in \bZ \setminus \{0\}$, or equivalently a unitary $u$ whose spectral measure is the Haar measure on the circle.\footnote{This sense of Haar unitary is not to be confused with the Haar random unitary matrix $U^{(n)}$ which is a random matrix chosen according to the Haar measure on the $n \times n$ unitary group.  For this reason, we will always refer to the latter as a Haar \emph{random} unitary.} We denote the set of Haar unitaries by $\mathcal{H}(\cM)$. Following \cite{patchellelayavalli2023sequential},  for $u, v \in \mathcal{H}(\cM)$, we say $u \sim_k v$ if there exist Haar unitaries $u = u_0, u_1, \dots, u_k = v$ in an ultrapower $\cM^{\cU}$ such that $[u_{j-1},u_j] = 0$ for $j = 1, \dots, k$.  We also write $u \sim v$ if $u \sim_k v$ for some $k$.
	
	The \emph{sequential commutation orbit} of $u$ is its equivalence class under the relation $\sim$.  As a consequence of \cite[Fact 2.9]{patchellelayavalli2023sequential}, the sequential commutation orbit of a Haar unitary $u$ is always contained inside the Pinsker algebra of $u$.  (Actually, in the special case of $L(F_\infty)$, the Pinsker algebra of some $\cA$ with $h(\cA:L(F_\infty)) = 0$ is equal to the algebra generated by its sequential commutation orbit \cite[Theorem 5.4]{patchellelayavalli2023sequential} as a consequence of the recent resolution of the Peterson-Thom conjecture of \cite[\S 7.5]{PetersonThom}, which occurred through a combination of $1$-bounded entropy techniques \cite{HayesPT} and strong convergence of tensor product random matrix models \cite{PTkilled,bordenave2023norm,dlSM2024,Parraud2024tensor,CGVvH2024strong2}.)
	
	\begin{introcor}[Freeness of sequential commutation orbits]\label{cor: freeness of SCorbit}
		Let $u_j(\omega)\in \mathcal{Q} = \prod_{n \to \cU} M_n(\mathbb{C})$  be as in the Theorem \ref{thm: commutant}.  Almost surely, the von Neumann algebras generated by the sequential commutation orbits of $u_1(\omega)$, $u_2(\omega)$, \dots respectively are freely independent.
	\end{introcor}
	
	Although Theorem \ref{thm: freeness of Pinsker} focuses on the matrix ultraproduct, we are able to transfer these results to $L(F_\infty)$ and more generally to arbitrary free products of Connes-embeddable von Neumann algebras through a natural argument using embeddings and countable saturation (a concept from model theory).
	
	\begin{introthm}[Freeness phenomena in ultrapowers of free products] \label{thm: free products}
		Let $(\cM_i)_{i \in I}$, \dots be diffuse Connes-embeddable tracial von Neumann algebras, and let $\cM = *_{i \in I} \cM_i$.  Let $\cV$ be a free ultrafilter on some index set $J$.  Then
		\begin{enumerate}
			\item If $\cA_i \subseteq \cM^{\cV}$ with $h(\cA_i: \cM^{\cV}) = 0$ and $\cA_i \cap \cM_i$ diffuse for each $i$, then the algebras $(\cM_i \vee \cA_i)_{i \in I}$ are freely independent.
			\item Let $\cC_i \subseteq \cM^{\cV}$ be the algebra generated by the sequential commutation orbits of Haar unitaries in $\cM_i$.  Then $(\cC_i)_{i \in I}$ are freely independent. In particular, if $u_i\in \cM_i$ are diffuse unitaries, then $\{u_i'\cap \cM^{\cV}\}_{i\in I}$ are freely independent. 
			\item Let $\cN_i \subseteq \cM^{\cV}$ be the von Neumann algebra generated by the wq-normalizer of $\cM_i$ in $\cM^{\cV}$ as defined in \cite{GalatanPopa}, namely,
			\[
			\cN_i = \mathrm{W}^*(\{u \in \mathcal{U}(\cM^{\cV}): u \cM_i u^* \cap \cM_i \text{ diffuse} \}).
			\]
			Then $(\cN_i)_{i \in I}$ are freely independent.
		\end{enumerate}
	\end{introthm}
	
	Note that this theorem is not stated in terms of Pinsker algebras because we have no assumptions on the $\cM_i$ other than Connes embeddability, so we have not assumed that $h(\cM_i) = 0$.
	In fact, in the proof, we embed $\cM_i$ into $\cR^{\cU}$ and study Pinsker algebras of these copies of $\cR^{\cU}$ inside their free product.  One can deduce that the algebras $\cA_i = \bigvee_{u\in \mathcal{H}(\cM_i)} \cP_{u}$ generated by the Pinsker algebras of Haar unitaries in $\cM_i$ are freely independent.  Alternatively, it would be natural to state the theorem in terms of Pinsker algebras \emph{relative} to $\cM_i$, that is, maximal algebras with entropy $0$ conditioned on $\cM_i$ (see \cite[\S 6]{JekelCoveringEntropy}), though for the sake of length we do not develop this perspective fully here.
	
	Theorem \ref{thm: free products} in particular gives a new proof and generalization of the freeness result for commutants of Houdayer-Ioana \cite[Theorem B]{houdayer2023asymptotic} in the case of Connes embeddable free products.  However, we cannot handle the case of amalgamated free products through this method.  Indeed, it is not even known if free products with amalgamation over a non-amenable algebra preserve Connes embeddability, and even if the amalgam is diffuse amenable, we expect similar issues to arise as in \cite[Remark 5.13]{freePinsker}.
	
	Theorem \ref{thm: free products} (1) also recovers \cite[Theorem A]{freePinsker}, which says that if $\cA_i$ is contained in $\cM$ with $\cA_i \cap \cM_i$ diffuse and $h(\cA_i:\cM) = 0$, then $\cA_i$ must be contained in $\cM_i$. Indeed, from Theorem \ref{thm: free products}, we see that $\cM_1\vee \cA_1$ is freely independent from $\mathcal{M}_2$ and this in particular means that $\cM_1\vee \cA_1= \cM_1$.  In Theorem \ref{thm: free products} (1), we do not assume $\cA_i$ is in $\cM$, only in $\cM^{\cU}$, but we still obtain free independence of $\cM_i \vee \cA_i$.  However, in this more general setting, one cannot conclude that $\cA_i \subseteq \cM_i^{\cU} \subseteq \cM^{\cU}$.  For instance, suppose that $\cM_i^{\cU}$ is commutative; then let $\cA_i$ be a copy of the hyperfinite $\mathrm{II}_1$ factor that intersects $\cM_i$ diffusely, which must exist since all Haar unitaries in $\cM^{\cU}$ are conjugate.  Then $\cA_i$ cannot be contained in the commutative algebra $\cM_i^{\cU}$.
	
	\subsection{Broader context and motivation}
	
	The broader motivation for our work includes a long history of results about the von Neumann algebra $L(F_\infty)$, going back to Murray and von Neumann's first papers in the subject \cite[\S 6]{MuvN43}.  The challenge of understanding its structure has been addressed with an astonishing diversity of tools, including not only the probabilistic methods that our paper draws on \cite{VoiculescuFreeEntropy3, GePrime, Jung2007,Hayes2018,HayesPT}, but also deformation rigidity theory \cite{PopaMaximalAmenable, Popasolidity, OzPo10, PopaVaesFree}; amenable actions, $C^*$ and boundary theory \cite{OzawaSolidActa, dkep2022properly, DP22};  closable derivations \cite{Peterson2009, PetersonDeriva}; free harmonic analysis and non-commutative $L_p$ space theory \cite{mei}, which was a key ingredient in Ioana and Houdayer's result \cite{houdayer2023asymptotic}.  We also point out that Popa showed the existence of elements in an ultraproduct freely independent from many subalgebras using incremental patching techniques \cite{Popa1995independence,Popa2014independence}.  On a related note, the existence of such freely independent elements in the $C^*$-ultrapower has recently led to developments in the classification of $C^*$-algebras, see \cite{louder2022strongly,robert2023selfless,amrutam2025strictcomparisonreducedgroup,tarski, hayes2025selflessreducedfreeproduct}.
	
	We were particularly motivated by the question of elementary equivalence of von Neumann algebras, and in particular of $L(F_\infty)$ and matrix ultraproducts.  The introduction of ultraproducts in functional analysis naturally inspired the classification question of when $\cM^{\cU}$ and $\cN^{\cV}$ are isomorphic for various ultrafilters $\cU$ and $\cV$.  Using mathematical logic, isomorphism of $\cM^{\cU}$ and $\cM^{\cV}$ for different ultrapowers depends on the continuum hypothesis \cite[Proposition 3.3]{FHS2013}.  But model theory also provides a powerful tool, known as the Keisler--Shelah theorem, which says that $\cM$ and $\cN$ admit some isomorphic ultrapowers if and only if they have the same first-order theory (see \cite[\S 2.2]{FHS2014b}), which here we understand in the sense of model theory for metric structures \cite{BYBHU2008}.  In this case, $\cM$ and $\cN$ are said to be \emph{elementarily equivalent}.
	
	The classification of tracial von Neumann algebras up to elementary equivalence is a challenging problem \cite[\S 4]{GH2023} and in particular very little is known in the setting of $\mathrm{II}_1$ factors without property Gamma.  For instance, we do not know if the matrix ultraproducts with different ultrafilters are elementarily equivalent \cite[\S 5]{FHS2014b} \cite[\S 5.2]{JekelModelEntropy}, we do not know if $L(F_m)$ and $L(F_n)$ are elementarily equivalent for $m \neq n$ \cite{GoldbringPi2023}, nor do we know if $L(F_\infty)$ is elementarily equivalent to a matrix ultraproduct \cite[Question 4.6]{GH2023}.  One could try to distinguish the theories of these algebras with certain first-order sentences relating to familiar properties such as commutation, free independence, and the like.  For instance, the construction in \cite{exoticCIKE} and with minor modifications in \cite[Theorems F and G]{houdayer2023asymptotic} and \cite[\S 2.4]{patchellelayavalli2023sequential} of non-Gamma $\mathrm{II}_1$ factors that are not elementarily equivalent is motivated by statements such as ``for all unitaries $u_1$ and $u_2$ with $u_1^2 = u_2^3 = 1$, there (approximately) exist Haar unitaries $v_1$ and $v_2$ with $[u_1,v_1] = [v_1,v_2] = [v_2,u_2] = 0$,'' see \cite[Remark 5.9]{exoticCIKE}; these results fit into the formalism of sequential commutation introduced in \cite{patchellelayavalli2023sequential}.  Houdayer and Ioana's theorem on freeness of commutants in \cite{houdayer2023asymptotic} implies that ``there exist unitaries $u_1$, $u_2$, \dots such that for all $b_1$, $b_2$, \dots with $[u_j,b_j] = 0$ (approximately), we have that $b_1$, $b_2$ \dots are (approximately) freely independent'' holds in $L(F_\infty)$, while our first result Theorem \ref{thm: commutant} shows that a similar statement holds in matrix ultraproducts; see Remark \ref{rem: theory} for a more precise statement.
	
	Moreover, Theorem \ref{thm: freeness of Pinsker} yields further first-order statements that hold in matrix ultraproducts; see Lemma \ref{lem: type implication}.  Since $u_1$, $u_2$, \dots arise from Haar random unitaries in the large-$n$ limit, our results also give some information about the values of first-order formulas on independent random unitaries.  While Voiculescu's asymptotic freeness theory \cite{VoicAsyFree, Voiculescu1998} describes the $*$-moments of random matrices, precious little information is known about formulas that involve quantifiers; for related ideas, see \cite[\S 5.2]{JekelModelEntropy}, \cite[\S 4]{FJP2023}.
	
	\subsection{Notation and Organization}
	
	Here \emph{tracial von Neumann algebra} refers to a von Neumann algebra $\cM$ with a fixed faithful normal tracial state $\tr^{\cM}$.  We write $\norm{x}_2 = \tr^{\cM}(x^*x)^{1/2}$.  In particular, when $\cM = M_n(\bC)$, we denote the normalized trace by $\tr_n$ and write $\norm{x}_2 = \tr_n(x^*x)^{1/2}$.  We assume that familiarity with basic theory of tracial von Neumann algebras (see for instance \cite{Sakai1971,Zhu1993,Bl06,anantharaman-popa}) as well as ultraproducts of tracial von Neumann algebras (see for instance \cite[Appendix A]{BO08}, \cite[\S 5.4]{anantharaman-popa}).
	
	Concerning the organization of the paper, we first give a self-contained proof of Theorems \ref{thm: random matrix commutant} and \ref{thm: commutant} through concentration and volumetric analysis in \S \ref{sec: commutants free}.  Then we prepare ingredients about model theory and $1$-bounded entropy in \S \ref{sec: tools}, which are subsequentily used in the proof of Theorems \ref{thm: freeness of Pinsker} and \ref{thm: free products} in \S \ref{sec: general version}.
	
	\subsection*{Acknowledgements}
	
	We are indebted to Ben Hayes for encouragement, generosity, and insightful discussions.
	We thank Adrian Ioana for bringing to our attention the problem of free independence of commutants in matrix ultraproducts.
	DJ thanks Jonathan Shi and Juspreet Singh for sharing a question about an application of random matrices in computer science that motivated uniform asymptotic freeness.
	We thank Greg Patchell and David Gao for various helpful discussions.
	We thank Yoonje Jeong as well as the anonymous referees for comments that improved the correctness and clarity of the paper.
	
	\section{Freeness of commutants} \label{sec: commutants free}
	
	In this section, after recalling some elementary facts about Haar random unitaries and diagonalization, we give our concentration of measure argument for uniform asymptotic freeness, and finally prove Theorems \ref{thm: random matrix commutant} and \ref{thm: commutant}.
	
	\subsection{Useful facts about diagonalization and approximate commutants} \label{subsec: diagonalization}
	
	In order to analyze the commutant of $u_j$ more easily, it will be convenient to diagonalize the Haar random unitary $U_j^{(n)}$.  In fact, up to a small error, we will be able to arrange that $U_j^{(n)}$ has the form $V_j^{(n)} A^{(n)} (V_j^{(n)})^*$ for deterministic $A^{(n)}$ with evenly spaced eigenvalues.  Thus, to prove the main results, we can perform our analysis on the model given by conjugates of a diagonal matrix, and then transfer them over to $U_j^{(n)}$'s.
	
	The following are useful elementary observations that can be considered folklore in random matrix theory, but we include the proofs here for completeness.  Some of the lemmas in this section will actually be used several times.
	
	\begin{lem} \label{lem: diagonal distance}
		For $k \in \bN$, let $[k] = \{1,\dots,k\}$.  For $j \in [k]$, let $I_{k,j}$ be the interval $[2\pi(j-1)/k, 2\pi j/k)$ viewed as a subset of the unit circle $S^1$.  Let
		\[
		\mathcal{O}_k = \{\mu \in \mathcal{P}(S^1): ~\forall j \in [k], ~ 1/k - 1/k^2 < \mu(I_{k,j}^\circ) \leq  \mu(\overline{I}_{k,j}) < 1/k + 1/k^2\}.
		\]
		Then $\mathcal{O}_k$ is an open neighborhood in $\mathcal{P}(S^1)$ of the Haar measure.
		
		Let $A^{(n)} = \diag(1, \zeta_n, \zeta_n^2, \dots, \zeta_n^{n-1})$, where $\zeta_n = e^{2\pi i/n}$.
		Let $B$ be a diagonal matrix with eigenvalues $e^{i \lambda_1}$, \dots, $e^{i \lambda_n}$ where $0 \leq \lambda_1 \leq \dots \leq \lambda_n < 2 \pi$.  If the empirical spectral distribution of $B$ is in $\mathcal{O}_k$, then $\norm{A^{(n)} - B} < 4 \pi / k$.
	\end{lem}
	
	\begin{proof}
		To see $\mathcal{O}_k$ that open, first note that $\mu(I_{k,j}^\circ) > 1/k - 1/k^2$ is an open condition in $\mathcal{P}(S^1)$.  This follows because it is the disjunction of the conditions $\int f\,d\mu > 1/k - 1/k^2$ for $f \in C(S^1)$ with $0 \leq f \leq \mathbf{1}_{I_{j,k}^{\circ}}$, since the indicator function of the open set is a supremum of a sequence of continuous functions.  Similarly, by taking complements, $\mu(\overline{I}_{k,j}) < 1/k + 1/k^2$ is an open condition.
		
		For the second claim, let $\mu$ be the empirical spectral distribution of $B$, and suppose $\mu \in \mathcal{O}_k$.  Then for $j \in [k]$,
		\[
		(j - 1)/k \leq j/k - j/k^2 \leq \mu(I_{k,1} \cup \dots \cup I_{k,j}) \leq j/k + j/k^2 \leq (j+1)/k.
		\]
		Therefore, if $t \in [n]$ with $t / n > (j+1)/k$, then $\lambda_t^{(n)} \geq 2\pi j/k$, and if $t \in [n]$ with $t /n < (j - 1)/k$, then $\lambda_t^{(n)} \leq 2\pi j/k$.  Overall, this implies that
		\[
		|\lambda_t^{(n)} - 2\pi t/n| \leq \frac{4 \pi}{k}, \text{ hence } |\zeta_t^{(n)} - e^{i\lambda_t^{(n)}}| \leq \frac{4\pi}{k}
		\]
		since the complex exponential function is $1$-Lipschitz.  Therefore, $\norm{B - A^{(n)}} \leq 4\pi/k$ as desired.
	\end{proof}

	\begin{prop} \label{prop: conjugation matrix model}
		Let $A^{(n)} = \diag(1, \zeta_n, \zeta_n^2, \dots, \zeta_n^{n-1})$, where $\zeta_n = e^{2\pi i/n}$.  Then there exists a family of independent Haar random unitaries $U_1^{(n)}$, \dots, $U_m^{(n)}$ and another family of independent Haar random unitaries $V_1^{(n)}$, \dots, $V_m^{(n)}$ on the same probability space $(\Omega,\mathcal{F},P)$ such that
		\[
		\lim_{n \to \infty} \sum_{j=1}^m \norm{U_j^{(n)} - V_j^{(n)} A^{(n)} (V_j^{(n)})^*} = 0 \text{ almost surely.}
		\]
	\end{prop}
	
	\begin{rem}
		The conclusion in the proposition shows that the operator norm goes to zero.  For this paper, we only need the weaker statement that the $2$-norm goes to zero.
	\end{rem}
	
	\begin{proof}[Proof of Proposition \ref{prop: conjugation matrix model}]
		First, we remark the following:  If $X$ is an $n \times n$ Haar random unitary and $Y$ is an $n \times n$ random unitary independent of $X$, then $XY$ and $YXY^*$ are Haar random unitaries.  In the case where $Y$ is deterministic, $Y X$ is a Haar unitary since the Haar measure is left-invariant, and then $YXY^*$ is a Haar unitary also since the Haar measure is right-invariant.  Now consider $Y$ that is random and independent of $X$.  Since $(X,Y)$ are independent, the joint distribution $(X,Y)$ has a disintegration given by conditioning on the value of $Y$.  Moreover, the conditional distributions of $XY$ and $YXY^*$ given that $Y$ is some fixed value $y$ are the Haar measure on the unitary group, which follows from the case of deterministic $Y$ handled above.  Now the distribution of $YX$ and $YXY^*$ respectively are obtained by integrating the conditional distributions given $Y = y$ with respect to the marginal distribution of $Y$.  Hence, they are also equal to the Haar measure on the unitary group.
		
		Now let $X_1^{(n)}$, \dots, $X_m^{(n)}$, $Y_1^{(n)}$, \dots, $Y_m^{(n)}$ be independent Haar random unitary matrices.  Let $U_j^{(n)} = Y_j^{(n)} X_j^{(n)}(Y_j^{(n)})^*$.  By the foregoing argument, $U_1^{(n)}$, \dots, $U_m^{(n)}$ are independent Haar random unitaries.  By the spectral theorem, we may write $X_j^{(n)} = W_j^{(n)} B^{(n)} (W_j^{(n)})^*$, where $B^{(n)}$ is diagonal and $W_j^{(n)}$ is unitary.  For $\zeta$ on the unit circle, let $\arg(\zeta)$ be the value of the argument that is in $[0,2\pi)$.  If $X_j^{(n)}$ has distinct eigenvalues, then there is a unique choice of $B_j^{(n)}$ where the arguments of the diagonal entries are in increasing order.  Moreover, in the case when $X_j^{(n)}$ has distinct eigenvalues (which happens almost surely), the choice of $W_j^{(n)}$ such that $X_j^{(n)} = W_j^{(n)} B^{(n)} (W_j^{(n)})^*$ is unique as well.  It is also straightforward to check that it depends on $X_j^{(n)}$ in a Borel-measurable manner.  Since the $X_j^{(n)}$'s are independent of the $Y_j^{(n)}$'s, we also have that $W_j^{(n)}$'s are independent of the $Y_j^{(n)}$'s, and therefore $V_j^{(n)} = Y_j^{(n)} W_j^{(n)}$ is a Haar random unitary, and of course $V_1^{(n)}$, \dots, $V_m^{(n)}$ are independent since $V_j^{(n)}$ only depends on $X_j^{(n)}$ and $Y_j^{(n)}$ for each $j$.  Overall, we have
		\[
		U_j^{(n)} = V_j^{(n)} B_j^{(n)} (V_j^{(n)})^*,
		\]
		where $U_1^{(n)}$, \dots, $U_m^{(n)}$ are independent Haar random unitaries, $V_1^{(n)}$, \dots, $V_m^{(n)}$ are independent Haar random unitaries (though not independent of $U_j^{(n)}$), and $B_j^{(n)}$ is diagonal with eigenvalues listed in order of increasing argument in $[0,2\pi)$.
		
		Note that
		\[
		\norm{U_j^{(n)} - V_j^{(n)} A_j^{(n)} (V_j^{(n)})^*}_2 = \norm{V_j^{(n)}B_j^{(n)} (V_j^{(n)})^* - V_j^{(n)} A_j^{(n)} (V_j^{(n)})^*}_2 = \norm{B_j^{(n)} - A_j^{(n)}}_2.
		\]
		Therefore, to complete the proof, it suffices to show that $\lim_{n \to \infty} \norm{B_j^{(n)} - A_j^{(n)}}_2 = 0$ almost surely.
		
		Let $\mu_j^{(n)}$ be the empirical spectral distribution of $U^{(n)}$, that is, the (random) probability measure on the circle that has a point mass of $1/n$ at each eigenvalue of $U_j^{(n)}$.  Note that empirical spectral distributions of $U_j^{(n)}$ and $B_j^{(n)}$ are the same.  By standard results about random unitary matrices (see for instance \cite[Theorem 4.13]{Meckes2019}), we have that almost surely $\mu_j^{(n)}$ converges weak-$*$ to the Haar measure on the unit circle.  Hence, almost surely, $\mu_j^{(n)}$ is eventually in the neighborhood $\mathcal{O}_k$ in Lemma \ref{lem: diagonal distance}.  It follows that $\norm{B_j^{(n)} - A^{(n)}}$ is eventually less than $4\pi / k$, and since $k$ is arbitrary, this completes the proof.
	\end{proof}
	
	This result on diagonalization enables us to reduce the study of approximate commutants of $U_j^{(n)}$ to the case of $A^{(n)}$.  In this setting, the approximate commutant is described by $\epsilon$-diagonal or band matrices.
	
	\begin{notn}
		For $i, j \in [n]$, let $d_n(i,j)$ denote the distance of $i$ and
		$j$ modulo $n$.  Let
		\[
		\mathcal{D}_\epsilon^{(n)} := \{B \in M_n(\bC): B_{i,j} = 0 \text{ when } d_n(i,j) > \epsilon n \}.
		\]
	\end{notn}
	
	\begin{lem} \label{lem: diagonal commutant}
		Let $A^{(n)} = \diag (1, \zeta_n, \zeta_n^2, \dots, \zeta_n^{n-1})$ where $\zeta_n = e^{2\pi i/n}$.  Let $B \in M_n(\mathbb{C})$.  Then for every $\epsilon > 0$, there exists $B_\epsilon \in \mathcal{D}_\epsilon^{(n)}$ with
		\[
		\norm{B - B_\epsilon}_2 \leq \frac{1}{\epsilon} \norm{[A^{(n)},B]}_2, \qquad \norm{B_\epsilon} \leq 3 \norm{B}.
		\]
	\end{lem}
	
	\begin{proof}
		Fix $\epsilon$ and $B$.  In the case where $\epsilon < 2/n$, we take $B_\epsilon$ to be the projection of $B$ onto diagonal matrices.  Note that
		\[
		\norm{B - B_\epsilon}_2^2 = \frac{1}{n} \sum_{j \neq k} |B_{j,k}|^2,
		\]
		while
		\begin{equation} \label{eq: diagonal commutator}
			\norm{[A^{(n)},B]}_2^2 = \frac{1}{n} \sum_{j \neq k} |\zeta_n^j - \zeta_n^k|^2 |B_{j,k}|^2.
		\end{equation}
		Note
		\[
		|\zeta_n^j - \zeta_n^k|^2 = |1 - \zeta_n^{j-k}|^2 = 2 - 2 \cos(2\pi d_n(j,k)/n) \geq 2 - 2 \cos(2 \pi/n)  
		\]
		Note that for $x \in [-\pi,\pi]$, we have $1 - \cos x = 2 \sin(x/2)^2 \geq  2x^2 / \pi^2$ since $|\sin(x/2)| \geq |x| / \pi$ on $[-\pi,\pi]$ using concavity of the sine function.  Hence,
		\[
		|\zeta_n^j - \zeta_n^k|^2 \geq \frac{4}{\pi^2} \left( \frac{2 \pi}{n} \right)^2 = \frac{16}{n^2} \geq 4 \epsilon^2.
		\]
		Overall,
		\[
		\norm{B - B_\epsilon}_2 \leq \frac{1}{2 \epsilon} \norm{[A^{(n)},B]}_2
		\]
		Now suppose that $\epsilon \geq 2/n$.  Let $m = \lfloor n\epsilon/2 \rfloor \geq 2$.  Write $n = qm + r$ for some $r \in \{0,\dots,m\}$.  For $j = 1$, \dots, $q$, let $P_j$ be the projection onto the basis vectors $e_{(j-1)m+1}$, \dots, $e_{jm}$; and let $P_{q+1}$ be the projection onto the last $r$ basis vectors.  Let
		\[
		B_\epsilon = \sum_{d_{q+1}(j,k) \leq 1} P_j B P_k.
		\]
		Note that all the indices $(j',k')$ where $(B_\epsilon)_{j',k'} \neq 0$ occur when $d_n(j',k') \leq 2m \leq n \epsilon$, and thus $B_\epsilon \in \mathcal{D}_\epsilon^{(n)}$.  Moreover, by writing
		\[
		B_\epsilon = \sum_{j=1}^{q+1} P_j B P_j + \sum_{j=1}^{q+1} P_j B P_{j+1} + \sum_{j=1}^{q+1} P_j B P_{j-1}
		\]
		(indices considered modulo $q+1$), we see that $\norm{B_\epsilon} \leq 3 \norm{B}$.  Next, note that the indices $(j',k')$ where $(B - B_\epsilon)_{j',k'} \neq 0$ must satisfy
		\[
		d_n(j',k') \geq m \geq n \epsilon / 2 - 1 \geq n \epsilon /4,
		\]
		and so
		\[
		2 - 2 \cos(2\pi d_n(j',k')/n) \geq 2 - 2 \cos(\pi \epsilon / 2) \geq \frac{4}{\pi^2} \left( \frac{\pi \epsilon}{2} \right)^2 = \epsilon^2.
		\]
		Therefore,
		\[
		\norm{B - B_\epsilon}_2^2 \leq \frac{1}{n} \sum_{d_n(j,k) > n \epsilon / 4} |B_{j,k}|^2 \leq \frac{1}{\epsilon^2} \frac{1}{n} \sum_{j \neq k} |\zeta_n^j - \zeta_n^k|^2 |B_{j,k}|^2 \leq \frac{1}{\epsilon^2} \norm{[A^{(n)},B]}_2^2.
		\]
	\end{proof}
	
	\begin{cor} \label{cor: ultraproduct commutant}
		Fix a free ultrafilter $\cU$ on $\bN$ and write $\cQ = \prod_{n \to \cU} M_n(\bC)$.  Let $A^{(n)} = \diag (1, \zeta_n, \zeta_n^2, \dots, \zeta_n^{n-1})$, and let $a = [A^{(n)}]_{n \in \bN} \in \cQ$. Let $V_j^{(n)}$ for $j \in \bN$ be independent Haar random unitaries.   Fix an outcome $\omega$ and let $v_j(\omega) = [V_j^{(n)}(\omega)] \in \cQ$.  Let
		\[
		\cD_\epsilon = \{[X^{(n)}]_{n \in \bN} \in \cQ: X^{(n)} \in \mathcal{D}_\epsilon^{(n)} \text{ for } n \in \bN \}.
		\]
		Then
		\[
		\{ v_j(\omega) a v_j(\omega)^*\}' \cap \cQ = \bigcap_{\epsilon > 0} v_j(\omega) \mathcal{D}_\epsilon v_j(\omega)^* 
		\]
	\end{cor}
	
	\begin{proof}
		Since commutants respect conjugation, it suffices to show that $\{a\}' \cap \cQ = \bigcap_{\epsilon > 0} \mathcal{D}_\epsilon$.  Let $b = [B^{(n)}]_{n \in \bN}$ be an element commuting with $a$, and fix $\epsilon > 0$.  Let $B_\epsilon^{(n)}$ be as in Lemma \ref{lem: diagonal commutant}.  Since $\norm{[A^{(n)},B^{(n)}]}_2 \to 0$ along the ultrafilter, we obtain that $\norm{B^{(n)} - B_\epsilon^{(n)}}_2 \to 0$ as well.  Hence, $b = [B_\epsilon^{(n)}]_{n \in \bN}$ is in $\mathcal{D}_\epsilon$, as desired.
	\end{proof}
	
	\subsection{Uniform asymptotic freeness via concentration of measure}
	
	For Theorems \ref{thm: commutant} and \ref{thm: random matrix commutant}, we will proceed as follows.  By Proposition \ref{prop: conjugation matrix model}, we can replace $U_j^{(n)}$ with $V_j^{(n)} A^{(n)} (V_j^{(n)})^*$.  In this section, we will argue that $V_j^{(n)} \mathcal{D}_\epsilon^{(n)} (V_j^{(n)})^*$ are asymptotically free as $n \to \infty$ up to some error tolerance depending on $\epsilon$.  Then since matrices that approximately commute with $V_j^{(n)} A^{(n)} (V_j^{(n)})^*$ will, for any $\epsilon$, be approximately in $V_j^{(n)} \mathcal{D}_\epsilon^{(n)} (V_j^{(n)})^*$, we will obtain the desired results.
	
	Hence, our present goal is to obtain a \emph{uniform} approximate asymptotic freeness result for $V_j^{(n)} \mathcal{D}_\epsilon^{(n)} (V_j^{(n)})^*$.  This is based on Voiculescu's famous asymptotic freeness theorem.
	
	\begin{thm}[Voiculescu's asymptotic freeness \cite{Voiculescu1998}] \label{thm: asymptotic freeness}
		Let $U_1^{(n)}$, \dots, $U_m^{(n)}$ be $n \times n$ Haar random unitaries.  Let $i_1 \neq i_2 \neq \dots i_k$.  Then
		\[
		\lim_{n \to \infty} \sup_{X_1, \dots, X_k \in B_1^{M_n(\bC)}} |\mathbb{E} \tr_n \left[U_{i_1}^{(n)}(X_1 - \tr_n(X_1))(U_{i_1}^{(n)})^* \dots U_{i_k}^{(n)}(X_k - \tr_n(X_k))(U_{i_k}^{(n)})^* \right]| = 0.
		\]
	\end{thm}
	
	\begin{proof}
		Let $B_r^{M_n(\bC)}$ denote the $r$-ball with respect to operator norm.  For each $n$, fix $X_1^{(n)}$, \dots, $X_k^{(n)}$ in $B_1^{M_n(\bC)}$ which maximize
		\begin{equation} \label{eq: moment to vanish}
			|\mathbb{E} \tr_n \left[U_{i_1}^{(n)}(X_1^{(n)} - \tr_n(X_1^{(n)}))(U_{i_1}^{(n)})^* \dots U_{i_k}^{(n)}(X_k^{(n)} - \tr_n(X_k^{(n)}))(U_{i_k}^{(n)})^* \right]|.
		\end{equation}
		A maximizer exists because the quantity depends continuously on the $X_j$'s.  It follows from \cite[Corollary 2.5]{Voiculescu1998} that \eqref{eq: moment to vanish} converges to $0$ as $n \to \infty$.
	\end{proof}
	
	In order to apply this result uniformly to all the matrices $X_j \in \mathcal{D}_\epsilon^{(n)}$, we rely on the high-dimensional concentration of measure phenomenon for Haar unitaries, which has been a staple of random matrix theory since \cite{BAG1997,GZ2000}.  We recall that the unitary group $\mathbb{U}_n$ (equipped with the Riemannian metric associated to the inner product $\ip{\cdot,\cdot}_{\Tr_n}$) satisfies the log-Sobolev inequality with constant $6/n$ \cite[Theorem 15]{Meckes2013}.  One can easily deduce the log-Sobolev inequality for the product of several copies of $\mathbb{U}_n$; see e.g.\ \cite[Corollary 5.7]{Ledoux2001}, \cite[Theorem 5.9]{Meckes2019}.  This in turn implies that it satisfies the Herbst concentration estimate; see e.g.\ \cite[Lemma 2.3.3]{AGZ2009}, \cite[Theorem 5.5]{Meckes2019}.  After renormalizing the metric to $\ip{\cdot,\cdot}_{\tr_n}$, one obtains the following concentration bound.  See \cite[\S 5.3]{freePinsker} for further explanation.
	
	\begin{lem}[Concentration] \label{lem: concentration}
		Let $U_1^{(n)}$, \dots, $U_m^{(n)}$ be $n \times n$ Haar random unitaries, and let $f: \mathbb{U}_n^m \to \mathbb{C}$ be Lipschitz with respect to $\norm{\cdot}_2$.  Then 
		\[
		P(|f(U_1^{(n)},\dots,U_m^{(n)}) - \mathbb{E}[f(U_1^{(n)},\dots,U_m^{(n)})]| \geq \delta) \leq 4 e^{-n^2 \delta^2 / 12 \norm{f}_{\Lip}^2}
		\]
	\end{lem}
	
	Concentration allows us to deduce the following uniform asymptotic freeness result, which is our main technical tool.  Here we replace $\mathcal{D}_\epsilon^{(n)}$ with a more general set $S_j^{(n)}$ that has relatively small covering numbers, which also depends on the index $j$ representing the position in the string or monomial; this added generality will be used in the proof of Lemma \ref{lem: type asymptotic freeness}.
	
	\begin{notn}
		Let $S$ be a subset of a metric space $X$.  Then $K_\epsilon(S)$ is defined as the smallest cardinality of a set $\Omega \subseteq X$ such that the $\epsilon$-neighborhood of $\Omega$ covers $S$.
	\end{notn}
	
	\begin{lem}[Uniform asymptotic freeness] \label{lem: uniform asymptotic freeness}
		Fix $k \in \bN$.  For each $n \in \bN$ and $j \in [k]$, let $S_j^{(n)} \subseteq B_1^{M_n(\bC)}$ and suppose that
		\begin{equation} \label{eq: covering assumption}
			\lim_{n \to \cU} \frac{1}{n^2} \log K_\epsilon(S_j^{(n)}) < \delta.
		\end{equation}
		Let $i_1 \neq i_2 \neq \dots i_k$ in $\bN$.  Then almost surely
		\begin{multline} \label{eq: uniform approximate freeness}
			\lim_{n \to \cU}
			\sup_{X_1 \in S_1^{(n)}} \dots \sup_{X_k \in S_k^{(n)}} \left|\tr_n \left[U_{i_1}^{(n)}(X_1 - \tr_n(X_1))(U_{i_1}^{(n)})^* \dots U_{i_k}^{(n)}(X_k - \tr_n(X_k))(U_{i_k}^{(n)})^* \right] \right| \\
			\leq 4k \epsilon + 2k \sqrt{12k \delta}.
		\end{multline}
		The same statement also holds when $\lim_{n \to \cU}$ is replaced by $\limsup_{n \to \infty}$ in both the hypothesis and the conclusion.
	\end{lem}
	
	\begin{proof}
		First, fix a set $\Omega_j^{(n)}$ such that the $\epsilon$-neighborhood of $\Omega_j^{(n)}$ covers $S_j^{(n)}$, and such that $|\Omega_j^{(n)}| = K_\epsilon^{(n)}(S_j^{(n)})$.  Although $\Omega_j^{(n)}$ is initially not assumed to be a subset of $S_j^{(n)}$, we can replace each element of $\Omega_j^{(n)}$ with an element from $S_j^{(n)}$ in its $\epsilon$-ball.  Hence, WLOG assume that $\Omega_j^{(n)} \subseteq S_j^{(n)}$ such that $S_j^{(n)}$ is in the $2 \epsilon$-neighborhood of $\Omega_j^{(n)}$.
		
		Let $m = \max(i_1,\dots,i_k)$.  Let
		\[
		f(X_1,\dots,X_k,U_1^{(n)},\dots,U_m^{(n)}) = \tr_n \left[U_{i_1}^{(n)}(X_1 - \tr_n(X_1))(U_{i_1}^{(n)})^* \dots U_{i_k}^{(n)}(X_k - \tr_n(X_k))(U_{i_k}^{(n)})^* \right]
		\]
		Since $\norm{X_j - \tr_n(X_j)} \leq \norm{X_j} \leq 1$, we see that $f$ is a $2k$-Lipschitz function of $(U_1^{(n)},\dots,U_m^{(n)})$ with respect to $\norm{\cdot}_2$.  In particular,
		\begin{multline} \label{eq: epsilon net}
			\sup_{X_1 \in S_1^{(n)}} \dots \sup_{X_k \in S_k^{(n)}} |f(X_1,\dots,X_k,U_1^{(n)},\dots,U_m^{(n)})| \\ \leq 4 k \epsilon + \sup_{X_1 \in \Omega_1^{(n)}} \dots \sup_{X_k \in \Omega_k^{(n)}} |f(X_1,\dots,X_k,U_1^{(n)},\dots,U_m^{(n)})|.
		\end{multline}
		Now by Lemma \ref{lem: concentration}, we have for each $(X_1,\dots,X_k)$ that
		\begin{multline}
			P(|f(X_1,\dots,X_k,U_1^{(n)},\dots,U_m^{(n)}) - \mathbb{E}[f(X_1,\dots,X_k,U_1^{(n)},\dots,U_m^{(n)})]| \geq 2k \sqrt{12 k \delta}) \\
			\leq 4e^{-n^2 12k\delta(2k)^2 / 12(2k)^2} = 4e^{-n^2 k \delta}
		\end{multline}
		Hence, by a union bound,
		\begin{multline*}
			P\left(\sup_{X_1 \in \Omega_1^{(n)}} \dots \sup_{X_k \in \Omega_k^{(n)}} |f(X_1,\dots,X_k,U_1^{(n)},\dots,U_m^{(n)}) - \mathbb{E}[f(X_1,\dots,X_k,U_1^{(n)},\dots,U_m^{(n)})]| \geq 2k \sqrt{12 k \delta} \right) \\
			\leq 4e^{-n^2 k \delta} \prod_{j=1}^k K_\epsilon(S_j^{(n)}).
		\end{multline*}
		By \eqref{eq: covering assumption}, there exists $\delta' < \delta$ and some $A \in \cU$ such that for all $n \in A$, for $j = 1$, \dots, $k$, we have $K_\epsilon(S_j^{(n)}) \leq e^{n^2 \delta'}$.  In particular, the previous equation is bounded by
		\[
		4e^{-n^2 k \delta} e^{n^2 k \delta'} = e^{-n^2k(\delta-\delta')}.
		\]
		Since $\delta' < \delta$, this quantity is summable over $n \in \bN$, and in particular, it is summable over $n \in A$.  Hence, by the Borel-Cantelli lemma, almost surely
		\begin{multline} \label{eq: concentration term bound}
			\lim_{\substack{n \to \infty \\ n \in A}} \sup_{X_1 \in \Omega_1^{(n)}} \dots \sup_{X_k \in \Omega_k^{(n)}} |f(X_1,\dots,X_k,U_1^{(n)},\dots,U_m^{(n)}) - \mathbb{E}[f(X_1,\dots,X_k,U_1^{(n)},\dots,U_m^{(n)})]| \\
			\leq 2k\sqrt{12k\delta}.    
		\end{multline}
		In particular, this holds for the limit as $n \to \cU$.  Moreover, by Theorem \ref{thm: asymptotic freeness}, we have that
		\begin{equation} \label{eq: expectation term bound}
			\lim_{n \to \infty} \sup_{X_1 \in \Omega_1^{(n)}} \dots \sup_{X_k \in \Omega_k^{(n)}} |\mathbb{E}[f(X_1,\dots,X_k,U_1^{(n)},\dots,U_m^{(n)})]| = 0.
		\end{equation}
		Combining \eqref{eq: epsilon net}, \eqref{eq: concentration term bound}, and \eqref{eq: expectation term bound} together with the triangle inequality implies \eqref{eq: uniform approximate freeness}.  The argument with $\limsup$ as $n \to \infty$ is the same (except that instead of choosing some $A \in \cU$, one just chooses $n$ sufficiently large).
	\end{proof}
	
	Finally, we record the following standard estimate on the covering number of $\mathcal{D}_\epsilon^{(n)}$.
	
	\begin{lem} \label{lem: diagonal covering}
		For $0 < \epsilon < R$, we have
		\[
		\frac{1}{n^2} \log K_\epsilon(\{X \in \mathcal{D}_\epsilon^{(n)}: \norm{X}_2 \leq R \}) \leq 2 \epsilon \log \frac{3R}{\epsilon}.
		\]
	\end{lem}
	
	\begin{proof}
		Note that $\dim_{\bR} \mathcal{D}_\epsilon^{(n)} \leq 2 \epsilon n^2$.  Let $\Omega$ be a maximal set of $\epsilon$-separated points in $S = \{X \in \mathcal{D}_\epsilon^{(n)}: \norm{X}_2 \leq R \}$.  Then the $\epsilon$-neighborhood of $\Omega$ covers $S$ by maximality.  The $\epsilon/2$ balls with centers in $\Omega$ are disjoint, and they are contained in the $R + \epsilon/2 \leq 3R/2$-ball centered at zero.  Therefore, $|\Omega|$ is at most
		\[
		\left( \frac{3R/2}{\epsilon/2} \right)^{2 \epsilon n^2} = \left( \frac{3R}{\epsilon} \right)^{2 \epsilon n^2}.  \qedhere
		\]
	\end{proof}

	\subsection{Proof of freeness of commutants}
	
	All the pieces are now in place to prove Theorems \ref{thm: random matrix commutant} and \ref{thm: commutant}.  We start with Theorem \ref{thm: commutant} first since it involves fewer approximation arguments.
	
	\begin{proof}[Proof of Theorem \ref{thm: commutant}]
		As in Proposition \ref{prop: conjugation matrix model}, let $U_j^{(n)}$ be independent Haar random unitaries.  Let $A^{(n)} = \diag(1,\zeta_n,\dots,\zeta_n^{n-1}$, where $\zeta_n = e^{2\pi i/n}$, and let $V_j^{(n)}$ be independent Haar random unitaries, such that $\norm{U_j^{(n)} - V_j^{(n)} A^{(n)} (V_j^{(n)})^*}_2 \to 0$ almost surely as $n \to \infty$.  Let $u_j(\omega) = [U_j^{(n)}(\omega)]_{n \in \bN} \in \cQ$, and let $v_j(\omega) = [V_j^{(n)}(\omega)]_{n \in \bN} \in \cQ$, and let $a = [A^{(n)}]_{n \in \bN} \in \cQ$.  Thus, almost surely $u_j(\omega) = v_j(\omega) a v_j(\omega)^*$, and so
		\[
		\{ u_j(\omega) \}' \cap \cQ = \{v_j(\omega) a v_j(\omega)^*\}' \cap \cQ = \bigcap_{\epsilon > 0} v_j(\omega) \mathcal{D}_\epsilon v_j(\omega)^*,
		\]
		where the second equality follows from Corollary \ref{eq: diagonal commutator} with $\mathcal{D}_\epsilon = \{[X^{(n)}]_{n \in \bN} \in \cQ: X^{(n)} \in \cD_\epsilon^{(n)} \}$.
		
		For each word $i_1 \neq \dots \neq i_k$, and for each $m \in \bN$, taking $\epsilon = 1/m$ in Lemma \ref{lem: diagonal covering} and $\delta = 2 \epsilon \log (3R/\epsilon)$ in Lemma \ref{lem: uniform asymptotic freeness}, we have that almost surely
		\begin{align} \label{eq: application of uniform asymptotic freeness}
			\limsup_{n \to \infty} \sup_{X_1 \in \mathcal{D}_{1/m}^{(n)}} \dots \sup_{X_k \in \mathcal{D}_{1/m}^{(n)}} \left|\tr_n \left[V_{i_1}^{(n)}(X_1 - \tr_n(X_1))(V_{i_1}^{(n)})^* \dots V_{i_k}^{(n)}(X_k - \tr_n(X_k))(V_{i_k}^{(n)})^* \right] \right| \\
			\leq 4k/m + 2k \sqrt{24k\log (3Rm)/m}.
		\end{align}
		There are only countably many words and values of $m$, so almost surely this holds for all $i_1$, \dots, $i_k$ and $m$, and also $\norm{U_j^{(n)} - V_j^{(n)} A^{(n)} (V_j^{(n)})^*}_2 \to 0$.  In the rest of the proof, we fix an outcome $\omega$ in this almost sure event.
		
		To show free independence of the commutants, fix $k \in \bN$, fix a word $i_1 \neq i_2 \neq \dots \neq i_k$, and for $j \in [k]$, fix $y_j \in \{u_{i_j}(\omega)\}' \cap \cQ$ with $\tr^{\cQ}(y_j) = 0$.  Let $x_j = v_{i_j}(\omega)^* y_j v_{i_j}(\omega)$, so that $x_j \in \{a\}' \cap \cQ$.  Then for each $m$, we can write $x_j = [X_j^{(n)}]_{n \in \bN}$ for some sequence $X_j^{(n)} \in \cD_{1/m}^{(n)}$ by Corollary \ref{cor: ultraproduct commutant}.  Then \eqref{eq: application of uniform asymptotic freeness} implies that
		\[
		\left|\tr^{\cQ} \left[v_{i_1}(\omega)x_1 v_{i_1}(\omega)^* \dots v_{i_k}(\omega)x_kv_{i_k}(\omega))^* \right] \right| \leq 4k/m + 2k \sqrt{24k\log (3Rm)/m}.
		\]
		Since $m$ was arbitrary, we obtain that $\tr^{\cQ}[y_1 \dots y_m] = 0$, and so we have shown free independence of $\{u_i(\omega)\}'$ for $i \in \bN$.
	\end{proof}
	
	For the case of Theorem \ref{thm: random matrix commutant}, we present here a slightly more general version with each $B_j^{(n)}$ replaced by a tuple.
	
	\begin{thm}[Asymptotic freeness of approximate commutants] \label{thm: random matrix commutants 2}
		Let $U_1^{(n)}$, $U_2^{(n)}$, \dots be independent $n \times n$ Haar random unitary matrices. For each $j$, let $\mathbf{B}_j^{(n)} = (B_{j,\ell}^{(n)})_{\ell \in \bN}$ be random matrices such that almost surely we have
		\[
		\limsup_{n \to \infty} \norm{B_{j,\ell}^{(n)}} < \infty
		\]
		and
		\[
		\lim_{n \to \infty} \norm{[B_{j,\ell}^{(n)}, U_j^{(n)}]}_2 = 0.
		\]
		Then almost surely the tuples $\mathbf{B}_j^{(n)}$ are asymptotically freely independent, that is, for each alternating word $i_1 \neq i_2 \neq \dots \neq i_k$ and any non-commutative polynomials $p_1$, \dots, $p_k$, we have almost surely
		\begin{equation} \label{eq: asymptotic free independence}
			\lim_{n \to \infty} \tr_n[(p_1(\mathbf{B}_{i_1}^{(n)}) - \tr_n[p_1(\mathbf{B}_{i_1}^{(n)})]) \dots (p_k(\mathbf{B}_{i_k}^{(n)}) - \tr_n[p_k(\mathbf{B}_{i_k}^{(n)})])] = 0.
		\end{equation}
	\end{thm}
	
	\begin{proof}
		Again, let $A^{(n)}$, $U_j^{(n)}$, and $V_j^{(n)}$ be as in Proposition \ref{prop: conjugation matrix model}.
		
		Fix the word $i_1$, \dots, $i_k$.  Let $C_j^{(n)} = p_j(\mathbf{B}_{i_j}^{(n)})$.  Observe that almost surely $\limsup_{n \to \infty} \norm{C_j^{(n)}} < \infty$.  Moreover, since $\limsup_{n \to \infty} \norm{B_{j,k}^{(n)}} < \infty$, we see that
		\[
		\limsup_{n \to \infty} \norm{[C_j^{(n)}, U_{i_j}^{(n)}]}_2 = \limsup_{n \to \infty} \norm{[p_j(\mathbf{B}_{i_j}^{(n)}), U_{i_j}^{(n)}]}_2 = 0,
		\]
		which follows from checking the case when $p_j$ is a monomial, which in turn follows from the triangle inequality and non-commutative H\"older inequality.
		
		Let $D_j^{(n)} = (V_{i_j}^{(n)})^* C_j^{(n)} V_{i_j}^{(n)}$, and note that almost surely $\norm{[D_j^{(n)},A^{(n)}]}_2 \to 0$.  Fix $m \in \bN$.  By Lemma \ref{lem: diagonal commutant}, there is a matrix $D_{j,\epsilon}^{(n)} \in \cD_{1/m}^{(n)}$ with
		\[
		\norm{D_{j,\epsilon}^{(n)}} \leq 3 \norm{D_j^{(n)}}, \qquad \norm{D_{j,\epsilon}^{(n)} - D_j^{(n)}}_2 \leq m \norm{[D_j^{(n)},A^{(n)}]}_2,
		\]
		and it is obvious from the proof of that lemma that $D_{j,\epsilon}^{(n)}$ is a measurable function on the probability space.  Recall that \eqref{eq: application of uniform asymptotic freeness} holds almost surely by Lemma \ref{lem: uniform asymptotic freeness} and \ref{lem: diagonal covering}.  Hence, almost surely
		\begin{multline*}
			\limsup_{n \to \infty} \left|\tr_n \left[V_{i_1}^{(n)}(D_{1,\epsilon}^{(n)} - \tr_n(D_{1,\epsilon}^{(n)}))(V_{i_1}^{(n)})^* \dots V_{i_k}^{(n)}(D_{k,\epsilon}^{(n)} - \tr_n(D_{k,\epsilon}^{(n)}))(V_{i_k}^{(n)})^* \right] \right| \\
			\leq 4k/m + 2k \sqrt{24k\log (3Rm)/m}.
		\end{multline*}
		Since $\norm{D_{j,\epsilon}^{(n)} - D_j^{(n)}}_2 \to 0$ almost surely and also $\limsup_{n \to \infty} \norm{D_{j,\epsilon}^{(n)}} < \infty$ almost surely, we obtain also that
		\begin{multline*}
			\limsup_{n \to \infty} \left|\tr_n \left[V_{i_1}^{(n)}(D_1^{(n)} - \tr_n(D_1^{(n)}))(V_{i_1}^{(n)})^* \dots V_{i_k}^{(n)}(D_k^{(n)} - \tr_n(D_k^{(n)}))(V_{i_k}^{(n)})^* \right] \right| \\
			\leq 4k/m + 2k \sqrt{24k\log (3Rm)/m}.
		\end{multline*}
		Then since $m$ was arbitrary and since $p_j(\mathbf{B}_{i_j}^{(n)}) = C_j^{(n)} = V_{i_j}^{(n)} D_j^{(n)} (V_{i_j}^{(n)})^*$, we obtain \eqref{eq: asymptotic free independence}.
	\end{proof}

	\section{Tools for the general approach} \label{sec: tools}
	
	Toward the proof of Theorems \ref{thm: freeness of Pinsker} and \ref{thm: free products}, we recall some results about model theory of operator algebras, as well as the version of Jung-Hayes $1$-bounded entropy \cite{Jung2007,Hayes2018} for types developed in \cite{JekelCoveringEntropy}.  The reason that we use full types in the proof of our main theorem and not just existential types (which would correspond to the entropy in the presence of Hayes) is explained in Remark \ref{rem: need Hausdorff}.
	
	\subsection{Model theory background}
	
	In the proof of Theorem \ref{thm: freeness of Pinsker} and Theorem \ref{thm: free products}, we will use several concepts from model theory of tracial von Neumann algebras, in particular \emph{formulas}, \emph{definable predicates}, \emph{types}, \emph{elementary submodels}, and \emph{countable saturation}.  We explain below the minimal background for these concepts in tracial von Neumann algebras.  For more general background on model theory for metric structures and tracial von Neumann algebras in particular, see \cite{BYBHU2008} \cite{Hart2023} \cite{FHS2014a}, \cite{GH2023}, \cite{Goldbring2023spectralgap}, \cite[\S 2-3]{JekelCoveringEntropy}, \cite[\S 2]{JekelDefinableClosure}.
	
	\textbf{Model theory of metric structures:} In model theory of metric structures \cite{BYBHU2008}, a certain category of objects is formalized through a \emph{language} $\cL$ which describes the operations (functions from $\cM^n$ to $\cM$ such as addition or multiplication) and predicates (functions from $\cM$ to $\bR$ such as the trace or the distance), which can then be used to state axioms for the structures of interest.  To state such axioms, one first defines \emph{formulas} as expressions in formal variables $x_1$, $x_2$, \dots built from the operations and predicates in the language together with \emph{connectives} and \emph{quantifiers}, as explained in more detail below.  Formulas with no free variables are called \emph{sentences}, and a collection of sentences is called a \emph{theory}.
	
	Farah, Hart, and Sherman \cite{FHS2014b} described the language $\cL_{\tr}$ for tracial von Neumann algebras, and a certain theory $\mathrm{T}_{\tr}$ that encodes the axioms of a tracial von Neumann algebra.  In general, \emph{$\cL$-structures} are metric spaces equipped with functions corresponding to the operations and predicates in $\cL$ (but which do not \emph{a priori} satisfy any particular list of axioms, or theory).  Thus, $\cL_{\tr}$-structures have formal operations of addition, multiplication, $*$, and trace, but do not necessarily satisfy the $*$-algebra axioms.  The $\cL_{\tr}$-structures that are actually tracial von Neumann algebras are precisely those which satisfy $\mathrm{T}_{\tr}$.
	
	\textbf{Formulas:} In continuous logic, the connectives are given by continuous functions.  The quantifiers are supremum and infimum over appropriate sets called \emph{sorts} or \emph{domains}; the domains for $\cL_{\tr}$ are the operator norm balls $B_r^{\cM}$.  $\cL_{\tr}$-formulas in variables $\mathbf{x} = (x_j)_{j \in \bN}$ are formal expressions built up recursively as follows:
	\begin{itemize}
		\item \emph{Basic formulas:} A \emph{basic formula} is an expression of the form $\re \tr(t(\mathbf{x}))$ or $\norm{t(\mathbf{x})}_2$, where $t$ is an expression formed through addition, multiplication, and $*$-operations (in practice, when evaluated on a von Neumann algebra, $t$ reduces to a $*$-polynomial \footnote{The definition of formula necessarily precedes writing axioms for von Neumann algebras. 
			Thus, in this definition we don't assume the $*$-algebra axioms, and hence, for instance, $x_1(x_2 + x_3)$ and $x_1x_2 + x_1 x_3$ are distinct expressions.}).
		\item \emph{Connectives:} If $\phi_1(\mathbf{x})$, \dots, $\phi_k(\mathbf{x})$ are formulas, and $F: \bR^k \to \bR$ is a continuous function or \emph{connective}, then $F(\phi_1(\mathbf{x}),\dots,\phi_k(\mathbf{x}))$ is a formula.
		\item \emph{Quantifiers:}  Let $\phi(\mathbf{x},y)$ be a formula in variables $\mathbf{x}$ and another variable $y$.  Fix $r > 0$ and recall $B_r$ denotes the operator-norm ball of radius $r$.  Then
		\[
		\sup_{y \in B_r} \phi(\mathbf{x},y) \text{ and } \inf_{y \in B_r} \phi(\mathbf{x},y)
		\]
		are formulas.
	\end{itemize}
	Given a formula $\phi$, a tracial von Neumann algebra $\cM$, and $x_1$, $x_2, \dots \in \cM$, we can evaluate $\phi^{\cM}(\mathbf{x})$ by substituting the actual elements $x_j$ instead of the formal variables, and evaluating each $\sup_{y \in B_r}$ or $\inf_{y \in B_r}$ symbol in the formula as the supremum or infimum over the ball $B_r^{\cM}$ in $\cM$.  The mapping $\phi^{\cM}: \cM^{\bN} \to \bR$ is called the \emph{interpretation} of the formula.
	
	An example of a formula for tracial von Neumann algebras is
	\[
	\phi(x) = \sup_{y \in B_1} \tr((xy - yx)^*(xy - yx)).
	\]
	For a tracial von Neumann algebra $\cM$ and $x \in \cM$, we have $\phi^{\cM}(x) = 0$ if and only if $x$ is in the center of $\cM$.
	
	\textbf{Definable predicates:} The set of formulas $\mathcal{F}$ forms an algebra over $\bR$ because the formulas can be added and multiplied (the addition and multiplication functions $\bR^2 \to \bR$ count as connectives).  From a functional-analytic point of view, it is natural to complete the space of formulas into a Banach algebra (or something similar); see \cite[\S 7.2]{Hart2023}.   In the case of $\cL_{\tr}$, we can obtain a ``Fr\'echet algebra'' by taking the completion with respect to uniform convergence on each operator norm ball.  More precisely, for a formula $\phi$ in variables $\mathbf{x} = (x_j)_{j \in \bN}$, for $\mathbf{r} = (r_j)_{j \in \bN}$, let
	\[
	\norm{\phi}_{\mathbf{r}} = \sup \left\{ |\phi^{\cM}(\mathbf{x})|: \cM \text{ tracial von Neumann algebra, } \mathbf{x} \in \prod_{j \in \bN} B_{r_j}^{\cM} \right\}.
	\]
	This defines a collection of seminorms, and the elements of the completion are called \emph{definable predicates relative to $\mathrm{T}_{\tr}$} (the notation here mentions $\mathrm{T}_{\tr}$ because we took the supremum only over $\cM$ which satisfy $\mathrm{T}_{\tr}$ rather than all $\cL_{\tr}$-structures).  We will later use the fact that definable predicates satisfy a certain uniform continuity property with respect to $2$-norm.
	
	\begin{lem}[{See \cite[Observation 3.11]{JekelCoveringEntropy}}] \label{lem: uniform continuity}
		Let $\phi$ be a definable predicate over $\mathcal{L}_{\tr}$ relative to $\mathrm{T}_{\tr}$ in countably many variables.  Let $\mathbf{r} \in (0,\infty)^{\bN}$ and $\epsilon > 0$.  Then there exists a finite $F \subseteq \bN$ and $\delta > 0$ such that, for every $\mathcal{M} \models \mathrm{T}_{\tr}$ and $\mathbf{x}, \mathbf{y} \in \prod_{j \in \bN} B_{r_j}^{\cM}$,
		\[
		d^{\cM}(x_j,y_j) < \delta \text{ for all } j \in F \implies |\phi^{\cM}(\mathbf{x}) - \phi^{\cM}(\mathbf{y})| < \epsilon.
		\]
	\end{lem}

	\textbf{Types:} Naturally, we want to identify the algebra of definable predicates with the algebra of continuous functions on its Gelfand spectrum.  This Gelfand spectrum is precisely the space of \emph{types}.  For each tuple $\mathbf{x}$, the \emph{type} of $\mathbf{x}$ is the mapping $\tp^{\cM}(\mathbf{x}): \phi \mapsto \phi^{\cM}(\mathbf{x})$ from the algebra of formulas (and more generally definable predicates) to the real numbers.  We denote the set of types of countable tuples in tracial von Neumann algebras by $\mathbb{S}(\mathrm{T}_{\tr})$.  For $\mathbf{r} \in (0,\infty)^{\bN}$, the set of types of tuples $\mathbf{x}$ with $x_j \in B_{r_j}$ will be denoted $\mathbb{S}_{\mathbf{r}}(\mathrm{T}_{\tr})$.  We equip $\mathbb{S}_{\mathbf{r}}(\mathrm{T}_{\tr})$ with the weak-$*$ topology (also known as the logic topology), and we equip $\mathbb{S}(\mathrm{T})$ with the inductive limit topology obtained by viewing it as the union of the spaces $\mathbb{S}_{\mathbf{r}}(\mathrm{T}_{\tr})$ (see \cite[\S 3.1]{JekelCoveringEntropy}).  With these definitions in hand, the algebra of definable predicates is isomorphic to $C(\mathbb{S}(\mathrm{T}_{\tr}))$ and the norm $\norm{\cdot}_{\mathbf{r}}$ given above coincides with the $C(\mathbb{S}_{\mathbf{r}}(\mathrm{T}_{\tr}))$ norm.
	
	The space $\mathbb{S}_{\mathbf{r}}(\mathrm{T}_{\tr})$ is compact Hausdorff and metrizable \cite[Observation 3.14]{JekelCoveringEntropy}.  In particular, for any point $\mu \in \mathbb{S}_{\mathbf{r}}(\mathrm{T}_{\tr})$, there is a nonnegative continuous function $f$ that vanishes only at $\mu$.  By \cite[Lemma 2.16]{JekelDefinableClosure}, $f$ automatically extends to a continuous function on $\mathbb{S}(\mathrm{T}_{\tr})$, that is, a definable predicate.  We record this result for later use.
	
	\begin{lem} \label{lem: isolate type}
		Let $\mu \in \mathbb{S}_{\mathbf{r}}(\mathrm{T}_{\tr})$ be a type.  Then there exists a nonnegative definable predicate $\phi$ such that whenever $\cM$ is a tracial von Neumann algebra and $\mathbf{x} \in \prod_{j \in \bN} B_{r_j}^{\cM}$, we have
		\[
		\phi^{\cM}(\mathbf{x}) = 0 \iff \tp^{\cM}(\mathbf{x}) = \mu.
		\]
	\end{lem}
	
	
	In the proof of Theorem \ref{thm: freeness of Pinsker}, we will be concerned with the limiting type of a countable tuple of independent Haar random unitary matrices, as well as the limiting types of deterministic matrices.  We first recall {\L}o{\'s}'s theorem in the setting of continuous logic for tracial von Neumann algebras.
	
	\begin{thm}[{See \cite[Theorem 5.4]{BYBHU2008}}] \label{thm: Los}
		Let $(\cM_i)_{i \in I}$ be tracial von Neumann algebras indexed by a set $I$.  Let $\cU$ be a non-principal ultrafilter on $I$.  Let $\mathbf{r} \in (0,\infty)^{\bN}$, and for each $i, j \in I$, let $x_{i,j} \in B_{r_j}^{\cM_i}$.  Let $\mathbf{x}_i = (x_{i,j})_{j \in \bN}$.  Let $\cM = \prod_{i \to \cU} \cM_i$ be the ultraproduct, and let $y_j$ be the equivalence class $[x_{i,j}]_{i \in I}$ in $\cM$, and let $\mathbf{y} = (y_j)_{j \in \bN}$.  Then for every formula $\phi$ in countably many variables, we have
		\[
		\phi^{\cM}(\mathbf{y}) = \lim_{i \to \cU} \phi^{\cM_i}(\mathbf{x}_i).
		\]
		Hence also $\lim_{i \to \cU} \tp^{\cM_i}(\mathbf{x}_i) = \tp^{\cM}(\mathbf{y})$.
	\end{thm}
	
	We will also rely on the fact that the limit of the type of a tuple of Haar random unitary matrices exists almost surely, which follows from concentration of measure.  The proof is the same as \cite[Lemma 4.2]{FJP2023}, so we only give a sketch below.  In the following, we denote the dual pairing of a type $\mu$ with a formula $\phi$ by $\mu[\phi]$.
	
	\begin{lem} \label{lem: Haar type}
		Fix a free ultrafilter $\cU$ on $\bN$.  Let $U_1^{(n)}$, $U_2^{(n)}$, \dots be independent $n \times n$ Haar random unitaries.  For each $\cL_{\tr}$-formula $\phi$ in countably many variables, let
		\[
		\mu_{\operatorname{Haar}}[\phi] = \lim_{n \to \cU} \mathbb{E} \phi^{M_n(\bC)}(U_1^{(n)},U_2^{(n)},\dots).
		\]
		Then almost surely $\lim_{n \to \cU} \phi^{M_n(\bC)}(U_1^{(n)},U_2^{(n)},\dots) = \mu_{\operatorname{Haar}}[\phi]$.  In other words,
		\[
		\tp^{M_n(\bC)}(U_1^{(n)},U_2^{(n)},\dots) \to \mu_{\operatorname{Haar}}
		\]
		weak-$*$ as $n \to \cU$ almost surely.
	\end{lem}
	
	\begin{proof}[Sketch of proof]
		Let $\mathbf{r} = (1,1,\dots)$.  Note that every $\cL_{\tr}$-formula $\phi$ can be approximated in $\norm{\cdot}_{\mathbf{r}}$ by one which is uniformly Lipschitz (for all tracial von Neumann algebras and all inputs); this follows by taking the connectives $F: \bR^k \to \bR$ to be Lipschitz; see \cite[proof of Lemma 4.2]{FJP2023}.  By a ``$3 \epsilon$ argument'' it suffices to check the claim when $\phi$ is Lipschitz.  In this case, using concentration of measure (Lemma \ref{lem: concentration}) and the Borel-Cantelli lemma, we see that almost surely
		\[
		\lim_{n \to \infty} |\phi^{M_n(\bC)}(U_1^{(n)},U_2^{(n)},\dots) - \mathbb{E} \phi^{M_n(\bC)}(U_1^{(n)},U_2^{(n)},\dots)| = 0,
		\]
		which implies the claim of the lemma.
	\end{proof}
	
	\begin{rem}
		Note that $\mu_{\Haar}$ may depend on the choice of ultrafilter $\cU$, since we do not even know whether the matrix ultraproducts for different ultrafilters are elementarily equivalent.  Very little is known at this point about the large $n$ behavior of formulas containing quantifiers on Haar unitaries.  See \cite[\S 5.2]{JekelModelEntropy} and \cite[\S 4]{FJP2023} for related discussion and results.
	\end{rem}
	
	\textbf{Elementary substructures:}  If $\cL$ is a metric language, and $\cM$ and $\cN$ are $\cL$-structures, we say that $\cM$ is an \emph{elementary submodel of $\cN$} if for every formula $\phi$ and for every tuple $\mathbf{x}$ in $\cM$, we have that $\phi^{\cM}(\mathbf{x}) = \phi^{\cN}(\mathbf{x})$.  Equivalently, $\tp^{\cM}(\mathbf{x}) = \tp^{\cN}(\mathbf{x})$.  We will use the following fact, known as the Downward L\"owenheim-Skolem theorem; see \cite[Proposition 7.3]{BYBHU2008}.  We state it for convenience in the particular case of tracial von Neumann algebras.
	
	\begin{prop}[Downward L\"owenheim-Skolem theorem]  \label{prop: DLS}
		Let $\cM$ be a tracial von Neumann algebra, and let $\cA \subseteq \cM$ be a separable von Neumann subalgebra.  Then there exists a separable elementary substructure $\widehat{\cM} \preceq \cM$ that contains $\cA$.
	\end{prop}
	
	\textbf{Countable saturation:}  Countable saturation of an $\cL$-structure means essentially that sets of formulas in countably many variables that admit approximate solutions must admit exact solutions.
	
	Let $\Phi$ be a set of $\cL_{\tr}$-formulas in countably many variables $(x_j)_{j \in \bN}$ and parameters or constants $(a_j)_{j \in \bN}$ from some tracial von Neumann algebra $\cM$.  Fix some $\mathbf{r} \in (0,\infty)^{\bN}$.
	\begin{itemize}
		\item We say that $\Phi$ is \emph{satisfiable in $\prod_{j \in \bN} B_{r_j}^{\cM}$} if there exists $\mathbf{x} \in \prod_{j \in \bN} B_{r_j}^{\cM}$ such that $\phi^{\cM}(\mathbf{x},\mathbf{a}) = 0$ for all $\phi \in \Phi$.
		\item We say that $\Phi$ is \emph{finitely approximately satisfiable in $\prod_{j \in \bN} B_{r_j}^{\cM}$} if for every $\phi_1$, \dots, $\phi_k$ in $\Phi$ and every $\epsilon$, there exists some $\mathbf{x} \in \prod_{j \in \bN} B_{r_j}^{\cM}$ satisfying $|\phi_i(\mathbf{x},\mathbf{a})| < \epsilon$ for $i = 1$, \dots, $k$.
		\item We say that $\cM$ is \emph{countably saturated} (or $\aleph_1$-saturated) if for every $\mathbf{r}$, and for every set of formulas $\Phi$ in countably many variables and countably many parameters $\mathbf{a}$, if $\Phi$ is finitely approximately satisfiable in $\prod_{j \in \bN} B_{r_j}^{\cM}$, then $\Phi$ is satisfiable in $\prod_{j \in \bN} B_{r_j}^{\cM}$.
	\end{itemize}
	
	Most ultraproducts are countably saturated.  Recall that an ultrafilter $\cU$ on a set $I$ is said to be \emph{countably incomplete} if there is a countable family of sets in $\cU$ with empty intersection.  We recall the following fact:
	
	\begin{lem}[See {\cite[Proposition 4.11]{FHS2014a}, \cite[Proposition 7.6]{BYBHU2008}}] \label{lem: saturation}
		Let $\cM = \prod_{i \to \cU} \cM_i$ for some tracial von Neumann algebras $(\cM_i)_{i \in I}$ where $I$ is an infinite index set and $\cU$ is a countably incomplete ultrafilter on $I$.  Then $\cM$ is countably saturated.
	\end{lem}

	\subsection{$1$-bounded entropy for types}
	
	In free entropy theory, $1$-bounded entropy is a notion of metric entropy for matricial approximations defined by Hayes \cite{Hayes2018} and inspired by the work of Jung \cite{Jung2007}.  Here we describe the version for full types from \cite{JekelCoveringEntropy}.
	
	If $\cO$ is a subset of the type space $\mathbb{S}(\mathrm{T}_{\tr})$ and $\mathbf{r} \in (0,\infty)^{\bN}$, we define
	\[
	\Gamma_{\mathbf{r}}^{(n)}(\cO) = \left\{\mathbf{X} \in \prod_{j \in \bN} D_{r_j}^{M_n(\bC)}: \tp^{M_n(\bC)}(\mathbf{X}) \in \cO \right\}.
	\]
	We view this as a microstate space as in Voiculescu's free entropy theory \cite{VoiculescuFreeEntropy2}. Entropy of types is defined by the exponential growth rates of covering numbers of these spaces $\Gamma_{\mathbf{r}}^{(n)}(\cO)$ up to unitary conjugation.
	
	\begin{defn}[Orbital covering numbers]
		Given $\Omega \subseteq M_n(\bC)^{\bN}$ and a finite $F \subseteq \bN$ and $\epsilon > 0$, we define $N_{F,\epsilon}^{\orb}(\Omega)$ to be the set of $\mathbf{Y} \in M_n(\bC)^{\bN}$ such that there exists a unitary $U$ in $M_n(\bC)$ and $\mathbf{X} \in \Omega$ such that $\norm{Y_i - UX_iU^*}_2 < \epsilon$ for all $i \in F$.  If $\Omega \subseteq N_{F,\epsilon}^{\orb}(\Omega')$, we say that $\Omega'$ \emph{orbitally $(F,\epsilon)$-covers} $\Omega$.  We denote by $K_{F,\epsilon}^{\orb}(\Omega)$ the minimum cardinality of a set $\Omega'$ that orbitally $(F,\epsilon)$-covers $\Omega$.
	\end{defn}
	
	\begin{defn}[{$1$-bounded entropy for types \cite{JekelCoveringEntropy}}]
		Fix a non-principal ultrafilter $\cU$ on $\bN$.  For  $\mu \in \mathbb{S}(\mathrm{T}_{\tr})$ and $F \subseteq I$ finite and $\epsilon > 0$, we define
		\[
		\Ent_{\mathbf{r},F,\epsilon}^{\mathcal{U}}(\mu) = \inf_{\text{open }\cO \ni \mu} \lim_{n \to \cU} \frac{1}{n^2} \log K_{F,\epsilon}^{\orb}(\Gamma_{\mathbf{r}}^{(n)}(\cO)).
		\]
		Then let
		\[
		\Ent_{\mathbf{r}}^{\cU}(\mu) := \sup_{\substack{\text{finite } F \subseteq \bN \\ \epsilon > 0} } \Ent_{\mathbf{r},F,\epsilon}^{\cU}(\mu).
		\]
		and
		\[
		\Ent^{\cU}(\mu) := \sup_{\mathbf{r} \in (0,\infty)^{\bN}} \Ent_{\mathbf{r}}^{\cU}(\mu).
		\]
	\end{defn}
	
	We also remark that the same definitions make sense for types of finite tuples $(x_1,\dots,x_m)$ instead of countable tuples.  In that setting, one does not need to consider the finite subset $F$ since $F$ can be taken to be $\{1,\dots,m\}$.
	
	By \cite[Corollary 4.10]{JekelCoveringEntropy}, if $\mathbf{x}$ and $\mathbf{y}$ are tuples in a tracial von Neumann algebra $\cM$ and $\mathrm{W}^*(\mathbf{x}) = \mathrm{W}^*(\mathbf{y})$, then $\tp^{\cM}(\mathbf{x})$ and $\tp^{\cM}(\mathbf{y})$ have the same entropy.  Hence, for separable $\cA \subseteq \cM$, one can define $\Ent^{\cU}(\cA: \cM)$ as the entropy of any generating tuple for $\cA$.  More generally, if $\cA$ is not necessarily separable, one defines
	\[
	\Ent^{\cU}(\cA:\cM) := \sup \{\Ent^{\cU}(\tp^{\cM}(\mathbf{x})): \mathbf{x} \in \cA^{\bN} \},
	\]
	and by \cite[Observation 4.12]{JekelCoveringEntropy} this agrees with the entropy of any countable (or finite) generating tuple in the case when $\cA$ is separable.  Now if $\cA \subseteq \cB \subseteq \cM$, then
	\begin{equation} \label{eq: monotonicity for type entropy}
		\Ent^{\cU}(\cA:\cM) \leq \Ent^{\cU}(\cB:\cM).
	\end{equation}

	\textbf{Entropy for types versus entropy in the presence:}  The relationship between the entropy for types described above and Hayes' $1$-bounded entropy is as follows (see \cite[\S 5]{JekelCoveringEntropy} for details).  Hayes' $1$-bounded entropy $h^{\cU}(\cA:\cM)$ of $\cA$ in the presence of $\cM$ (with respect to the ultrafilter $\cU$) arises by looking at microstate spaces for the \emph{existential type} of tuples $\mathbf{x}$ rather than the full type.  The existential type describes the evaluation on $\mathbf{x}$ of formulas of the form $\phi(\mathbf{x}) = \inf_{y \in B_r} \psi(\mathbf{x},y)$, where $\psi$ is a quantifier-free formula (i.e.\ a formula with no supremum or infimum in it).  The space of existential types is equipped with a topology that is non-Hausdorff because a basic neighborhood is defined by one-sided upper bounds on a finite family of inf-formulas.  The microstate spaces defined by a neighborhood of the existential type are a special case of microstate spaces defined by neighborhoods of the full type, and hence
	\begin{equation} \label{eq: existential entropy bound}
		\Ent^{\cU}(\cA: \cM) \leq h^{\cU}(\cA: \cM);
	\end{equation}
	see e.g.\ \cite[Lemma 5.13]{JekelCoveringEntropy}.
	
	\textbf{Orbital versus relative entropy:} There are actually two approaches to defining metric entropy--one based on covering numbers up to unitary conjugation, and one based on covering number of microstate spaces relative to a fixed microstate sequence $A^{(n)}$ for a normal element $a$ with diffuse spectrum.  These were shown to be equivalent in \cite[Lemma A.5]{Hayes2018} for the setting of $a$ self-adjoint and for the $1$-bounded entropy in the presence $h$.  We need the analogous result for the entropy for types, and with using a Haar unitary instead of a self-adjoint element for the fixed microstate.  We include a self-contained proof for convenience; the approach here is slightly different since we fix a very specific $A^{(n)}$ and use Lemma \ref{lem: diagonal commutant} rather than using Szarek's covering estimates for Grassmannians \cite{Szarek}.
	
\begin{lem} \label{lem: relative entropy}
	Let $\mu \in \mathbb{S}_{\mathbf{r}}(\mathrm{T}_{\tr})$ be the type of some infinite tuple $(a,x_1,x_2,\dots)$ such that $a$ is a Haar unitary.  Let $A^{(n)} = \diag(1,\zeta_n,\dots,\zeta_n^{n-1})$, where $\zeta_n = e^{2\pi i/n}$.  For a neighborhood $\cO$ of $\mu$, let
	\[
	\Gamma_{\mathbf{r}}^{(n)}(\cO \mid A^{(n)} \rightsquigarrow a) = \{\mathbf{X}^{(n)} \in \prod_{j \in \bN} D_{r_j}^{M_n(\bC)}: \tp^{M_n(\bC)}(A^{(n)},\mathbf{X}^{(n)}) \in \cO \}.
	\]
	Then
	\[
	\Ent_{\mathbf{r}}^{\cU}(\mu) = \sup_{(F,\epsilon)} \inf_{\cO \ni \mu} \lim_{n \to \cU} \frac{1}{n^2} \log K_{F,\epsilon}(\Gamma_{\mathbf{r}}^{(n)}(\cO \mid A^{(n)} \rightsquigarrow a)).
	\]
\end{lem}

\begin{proof}
	First, let us show $\leq$.  Given $(F,\epsilon)$ and given a neighborhood $\cO$ of $\mu$, we claim that there is a neighborhood $\cO'$ such that $\Gamma_{(1,\mathbf{r})}(\cO')$ is contained in the orbital $(F,\epsilon/2)$-neighborhood of $\{A^{(n)}\} \times \Gamma_{\mathbf{r}}(\cO \mid A^{(n)} \rightsquigarrow a)$.
	
	By Lemma \ref{lem: isolate type}, there exists a definable predicate $\phi$ such that, when $\cM$ is a tracial von Neumann algebra and $\mathbf{y} \in \prod_{j \in \bN} B_{r_j}^{\cM}$, then $\phi^{\cM}(b,\mathbf{y}) = 0$ if and only if $\tp^{\cM}(b,\mathbf{y}) = \mu$.  Note that there exists $m \in \bN$ such that if $\phi^{\cM}(b,\mathbf{y}) \leq 1/m$, then $\tp^{\cM}(\mathbf{y}) \in \cO$.  Indeed, let $\mathcal{K}_m$ be the set of types in $\mathbb{S}_{\mathbf{r}}(\mathrm{T}_{\tr}) \setminus \cO$ that satisfy $\varphi \leq 1/m$, which is a closed subset of $\mathbb{S}_{\mathbf{r}}(\mathrm{T}_{\tr})$.  If $\mathcal{K}_m$ were nonempty for all $m$, then by compactness $\bigcap_{m \in \bN} \mathcal{K}_m$ would be nonempty, which would yield some $\cM$ and $(b,\mathbf{y})$ such that $\phi^{\cM}(b,\mathbf{y}) = 0$ and $\tp^{\cM}(b,\mathbf{y}) \not \in \cO$, which contradicts the choice of $\phi$.  Hence, we have that $\phi^{\cM}(b,\mathbf{y}) \leq 1/m$ implies $\tp^{\cM}(b,\mathbf{y}) \in \cO$.
	
	Using uniform continuity (Lemma \ref{lem: uniform continuity}), there exists $\delta > 0$ such that for all tracial von Neumann algebras $\cM$ and for $a, b \in B_{r_1}^{\cM}$ and $x_j \in B_{r_j}^{\cM}$, $j \geq 2$, we have
	\[
	\norm{a - b}_2 < \delta \implies
	|\phi^{\cM}(a,\mathbf{x}) - \phi^{\cM}(b,\mathbf{x})| < \frac{1}{2m}.
	\]
	Let $k$ be such that $4 \pi /k < \min(\delta,\epsilon/2)$, and let $\mathcal{O}_k$ be the neighborhood of the Haar measure on the unit circle described in Lemma \ref{lem: diagonal distance}; then since the weak-$*$ topology on $\mathcal{P}(S^1)$ is given by testing against trigonometric polynomials, there exist $*$-polynomials $p_1$, \dots, $p_\ell$ such that
	\begin{itemize}
		\item $\re \tr(p_j(a)) < 0$ for all $j$;
		\item for any unitary $u$, if $\re \tr(p_j(u)) < 0$, for all $j$, then the spectral distribution of $u$ is in $\cO_k$.
	\end{itemize}
	
	Now let $\cO'$ be the set of types $\tp^{\cM}(b,\mathbf{y})$ satisfying
	\[
	\phi^{\cM}(b,\mathbf{y}) < \frac{1}{2m} \text{ and } \max_{j \in [\ell]} \re \tr(p_j(b)) < 0.
	\]
	Suppose that $(B,\mathbf{Y})$ is an $n \times n$ matrix tuple whose type is in $\mathcal{O}'$.  Then the spectral measure of $B$ is in $\cO_k$, and therefore by Lemma \ref{lem: diagonal distance}, there is a unitary $V$ such that $\norm{V^*BV - A^{(n)}}_2 < 4 \pi / k < \min(\delta,\epsilon/2)$.  Let $\mathbf{X} = V \mathbf{Y} V^*$.  Since $\norm{A^{(n)} - VBV^*}_2 < \delta$, we have
	\[
	\phi^{M_n(\bC)}(A^{(n)},\mathbf{X}) < \phi^{M_n(\bC)}(VBV^*,\mathbf{X}) + \frac{1}{2m} = \phi^{M_n}(B,\mathbf{Y}) + \frac{1}{2m} < \frac{1}{m},
	\]
	where we used the unitary invariance of the formula $\phi$ and the assumption that $\tp^{M_n(\bC)}(B,\mathbf{Y}) \in \cO'$.  Therefore, $\tp^{M_n(\bC)}(A^{(n)},\mathbf{X}) \in \cO$, and $(B,\mathbf{Y})$ is in the orbital $\epsilon/2$-neighborhood of $(A^{(n)},\mathbf{X})$.  Thus, $\Gamma_{\mathbf{r}}^{(n)}(\cO)$ is in the orbital $(F,\epsilon/2)$ neighborhood of $\{A^{(n)}\} \times \Gamma_{\mathbf{r}}^{(n)}(\cO \mid A^{(n)} \rightsquigarrow a)$.  Hence,
	\[
	K_{F,\epsilon}^{\orb}(\Gamma_{\mathbf{r}}^{(n)}(\cO')) \leq K_{F,\epsilon/2}(\Gamma_{\mathbf{r}}^{(n)}(\cO \mid A^{(n)} \rightsquigarrow a)).
	\]
	From this we obtain
	\[
	\Ent_{\mathbf{r},F,\epsilon}^{\cU}(\mu) \leq \inf_{\cO \ni \mu} \lim_{n \to \cU} \frac{1}{n^2} \log K_{F,\epsilon/2}(\Gamma_{\mathbf{r}}^{(n)}(\cO \mid A^{(n)} \rightsquigarrow a)),
	\]
	and then taking the supremum over $(F,\epsilon)$ finishes the inequality $\leq$.
	
	For the other direction, again fix $(F,\epsilon)$ and $\cO$.  We want to bound the $(F,\epsilon)$-covering number of $\Gamma_{\mathbf{r}}^{(n)}(\cO \mid A^{(n)} \rightsquigarrow a)$ in terms of the orbital $(F,\delta)$-covering number of $\Gamma_{\mathbf{r}}^{(n)}(\cO)$ where $\delta \in (0,1)$ is to be chosen later.  Let $\Omega$ be a set that orbitally $(F,\delta)$-covers $\Gamma_{\mathbf{r}}^{(n)}(\cO)$.  In particular, $\{A^{(n)}\} \times \Gamma_{\mathbf{r}}^{(n)}(\cO \mid A^{(n)} \rightsquigarrow a)$ is covered by the orbital neighborhoods $N_{F,\delta}^{\orb}(B,\mathbf{X})$ for $(B,\mathbf{X}) \in \Omega$.  Thus, our goal is to estimate the plain covering number of the sets
	\[
	S_{(B,\mathbf{X})} = [\{A^{(n)}\} \times \Gamma_{\mathbf{r}}^{(n)}(\cO \mid A^{(n)} \rightsquigarrow a)] \cap N_{F,\delta}^{\orb}(B,\mathbf{X}).
	\]
	The idea is that unitaries that we conjugate by must approximately fix $A^{(n)}$ and hence are approximately band matrices.  More precisely, if $(A^{(n)},X_1)$ and $(A^{(n)},X_2)$ are both in $S_{(B,\mathbf{X})}$, then there is a unitary $V$ such that
	\[
	(A^{(n)},X_1) \in N_{F,\delta}(VA^{(n)}V^*,VX_2V^*).
	\]
	In particular, $\norm{[V,A^{(n)}]}_2 < \delta$, and so by Lemma \ref{lem: diagonal commutant},
	\[
	d(V, \mathcal{D}_{\sqrt{\delta}}^{(n)} \cap B_3^{M_n(\bC)}) \leq \frac{8\sqrt{\pi}}{\sqrt{\delta}} 2 \delta = 16 \sqrt{\pi \delta}.
	\]
	Thus, by Lemma \ref{lem: diagonal covering},
	\[
	K_{F,\sqrt{\delta} + 16 \sqrt{\pi \delta}}( \{V \in \mathbb{U}_n: \norm{[V,A^{(n)}]}_2 < \delta \} ) \leq K_{F,\sqrt{\delta}}(\mathcal{D}_{\sqrt{\delta}}^{(n)} \cap B_3^{M_n(\bC))}) \leq \exp(n^2 \cdot 2 \sqrt{\delta} \log(3R/\sqrt{\delta}) ).
	\]
	Hence, fixing some $(A^{(n)},X_0) \in S_{(B,\mathbf{X})}$, we have that
	\[
	S_{(B,\mathbf{X})} \subseteq N_{F,2\delta}( \{(A^{(n)},VX_0V^*):  V \in \mathbb{U}_n: \norm{[V,A^{(n)}]}_2 < \delta \} ).
	\]
	Then since $\norm{(X_0)_j} \leq r_j$, we have that $V \mapsto V(X_0)_j V^*$ is $2r_j$-Lipschitz.  Let $r_F = \max_{j \in F} r_j$. Then any $(F,\sqrt{\delta} + 16 \sqrt{\pi\delta})$-covering of $\{V \in \mathbb{U}_n: \norm{[V,A^{(n)}]}_2 < \delta \}$ yields a $(F,r_F(\sqrt{\delta} + 16 \sqrt{\pi\delta}))$-covering of the conjugation orbit of $(A^{(n)},X_0)$ and so a $(F, 2 \delta + r_F(\sqrt{\delta} + 16 \sqrt{\pi\delta}))$-covering of $S_{(B,\mathbf{X})}$.  Hence, choose $\delta$ small enough that $2 \delta + r_F(\sqrt{\delta} + 16 \sqrt{\pi\delta}) < \epsilon$.  Then we obtain
	\[
	K_{F,\epsilon}(\Gamma_{\mathbf{r}}^{(n)}(\cO \mid A^{(n)} \rightsquigarrow a)) \leq \sum_{(B,\mathbf{X}) \in \Omega} K_{F,\epsilon}(S_{(B,\mathbf{X})}) \leq K_{F,\delta}^{\orb}(\Gamma_{\mathbf{r}}^{(n)}(\cO)) \exp(n^2 \cdot 2 \sqrt{\delta} \log(3R/\sqrt{\delta}) ).
	\]
	Hence,
	\[
	\frac{1}{n^2} \log K_{F,\epsilon}(\Gamma_{\mathbf{r}}^{(n)}(\cO \mid A^{(n)} \rightsquigarrow a)) \leq \frac{1}{n^2} \log K_{F,\delta}^{\orb}(\Gamma_{\mathbf{r}}^{(n)}(\cO)) + 2 \sqrt{\delta} \log(3R/\sqrt{\delta}).
	\]
	Since $\delta$ can be chosen arbitarily small, we get the desired inequality.
\end{proof}
	
	\section{General freeness results} \label{sec: general version}
	
	We are now ready to prove Theorem \ref{thm: freeness of Pinsker} and \ref{thm: free products}.
	
	\subsection{Free independence phenomena in matrix ultraproducts}
	
	In this section, we fix a countably incomplete ultrafilter $\cU$ on $\bN$, and write $\cQ = \prod_{n \to \cU} M_n(\bC)$.
	
	\begin{lem} \label{lem: type asymptotic freeness}
		Let $U_1^{(n)}$, $U_2^{(n)}$, \dots be independent Haar unitaries.  Let $k \in \bN$, and let $i_1 \neq i_2 \neq \dots \neq i_k$.  For each $j = 1$, \dots, $k$, let $\nu_j \in \mathbb{S}_{(1,1)}(\mathrm{T}_{\tr})$ be the type of some pair $(u,x)$ in $\cQ$ where $u$ is a Haar unitary and $x$ has trace zero, such that $\Ent^{\cU}(\nu_j) = 0$.  Then almost surely, for all $x_j$ with $\tp^{\cQ}(u_{i_j}(\omega),x_j) = \nu_j$, we have $\tr^{\cQ}(x_1 \dots x_k) = 0$.
	\end{lem}
	
	\begin{proof}
		By Proposition \ref{prop: conjugation matrix model}, let $U_j^{(n)}$ be an independent family of Haar random unitaries and $V_j^{(n)}$ another independent family of Haar random unitaries such that $\norm{U_j^{(n)} - V_j^{(n)} A^{(n)} (V_j^{(n)}})^* \to 0$ almost surely, where $A^{(n)} = \diag(1,\zeta_n,\dots,\zeta_n^{n-1})$.  For each outcome $\omega$, let $v_j(\omega)$ be the corresponding element of the ultraproduct $\cQ$, so that $u_j(\omega) = v_j(\omega) a v_j(\omega)^*$.
		
		Fix $m \in \bN$.  Since $\Ent^{\cU}(\nu_j) = 0$, using Lemma \ref{lem: relative entropy} for the case of a $1$-tuple rather than a countable tuple, there exists some neighborhood $\cO_{j,m}$ of $\nu_j$ such that
		\[
		\lim_{n \to \cU} \frac{1}{n^2} \log K_{1/m}(\Gamma_1^{(n)}(\cO_{j,m} \mid A^{(n)} \rightsquigarrow a)) < \frac{1}{m}.
		\]
		(Here we take $\epsilon = 1/m$ and $F = \{1\}$.)  Then applying Lemma \ref{lem: uniform asymptotic freeness} with $\delta = \epsilon = 1/m$, we see that almost surely 
		\begin{multline*}
			\lim_{n \to \cU}
			\sup_{X_1 \in \Gamma_1^{(n)}(\cO_{1,m} \mid A^{(n)} \rightsquigarrow a)} \dots \sup_{X_k \in \Gamma_1^{(n)}(\cO_{k,m} \mid A^{(n)} \rightsquigarrow a)} \\ \left|\tr_n \left[V_{i_1}^{(n)}(X_1 - \tr_n(X_1))(V_{i_1}^{(n)})^* \dots V_{i_k}^{(n)}(X_k - \tr_n(X_k))(V_{i_k}^{(n)})^* \right] \right| \\
			\leq 4k/m + 2k \sqrt{12k/m}.
		\end{multline*}
		By taking the countable intersection, almost surely this holds for \emph{all} $m \in \bN$.
		
		Now fix an outcome $\omega$ in the almost sure event where the above inequality holds for all $m$.  Suppose that $\tp^{\cQ}(u_{i_j}(\omega),y_j) = \nu_j$.  Write $y_j = v_{i_j}(\omega) x_j v_{i_j}(\omega)^*$, and represent $X_j = [X_j^{(n)}]_{n \in \bN}$ with $\norm{X_j^{(n)}} \leq 1$.  By {\L}o{\'s}'s theorem (Theorem \ref{thm: Los}),
		\[
		\lim_{n \to \cU} \tp^{M_n(\bC)}(A^{(n)}, x_j^{(n)}) = \tp^{\cQ}(a,x_j) = \tp^{\cQ}(u_{i_j}(\omega),y_j) = \nu_j.
		\]
		Thus, there is some $A \in \cU$ such that $\tp^{M_n(\bC)}(A^{(n)},X_j^{(n)}) \in \cO_{m,j}$ for each $j = 1$, \dots, $k$ for each $n \in A$.  Hence,
		\begin{multline*}
			|\tr^{\cQ}(y_1 \dots y_k)| = |\tr^{\cQ}(v_{i_1}(\omega) x_1 v_{i_1}(\omega)^* \dots v_{i_k}(\omega) x_k v_{i_k}(\omega)^*)| \\ = \lim_{n \to \cU} \left|\tr_n \left[V_{i_1}^{(n)}(\omega)(X_1 - \tr_n(X_1))(V_{i_1}^{(n)}(\omega))^* \dots V_{i_k}^{(n)}(\omega)(X_k - \tr_n(X_k))(V_{i_k}^{(n)}(\omega))^* \right] \right| \\ \leq 4k/m + 2k \sqrt{12k/m}.    
		\end{multline*}
		Then since $m$ was arbitrary, we get $\tr^{\cQ}(y_1 \dots y_k) = 0$ as desired.
	\end{proof}
	
	In the proof of Theorem \ref{thm: vanishing entropy freeness} below, we will want to apply the conclusion of Lemma \ref{lem: type asymptotic freeness} simultaneously to \emph{all} $\nu_j$'s satisfying the hypotheses (not just to a countable family of such $\nu_j$'s).  So we will not be able to do this by taking intersections of almost sure events na\"\i vely.  Rather, we proceed by arguing that the conclusion of the lemma is a property of the \emph{type} of $(u_1(\omega),u_2(\omega),\dots)$.  We already know that the type of the Haar unitaries converges almost surely to $\mu_{\Haar}$ (there we only had to test countably many formulas).  We will show that for each $\nu_1$, \dots, $\nu_k$ as in Lemma \ref{lem: type asymptotic freeness}, the conclusion of the lemma can be expressed in terms of formulas in $(u_1,u_2,\dots)$, and so it depends only the type.  Therefore, since the type $\mu_{\Haar}$ can be described only by testing countably many formulas, one can indeed obtain the conclusion of Lemma \ref{lem: type asymptotic freeness} for uncountably many choices of $(\nu_1,\dots,\nu_k)$.
	
	The next lemma follows from Lemma \ref{lem: type asymptotic freeness} from purely model-theoretic considerations, specifically the countable saturation of the matrix ultraproduct $\cQ$.  Compare for instance \cite[proof of Corollary 9.10]{BYBHU2008}.  In the following, a \emph{modulus of continuity} refers to a continuous increasing function $w: [0,\infty) \to [0,\infty)$ such that $w(0) = 0$.
	
	\begin{lem} \label{lem: type implication}
		Let $\mu_{\Haar} \in \mathbb{S}_{(1,1,\dots)}(\operatorname{Th}(\cQ))$ be the almost sure limit of the type of Haar random unitary matrices as $n \to \cU$.  Let $k \in \bN$, and let $i_1 \neq i_2 \neq \dots \neq i_k$.  For each $j = 1$, \dots, $k$, let $\nu_j \in \mathbb{S}_{(1,1)}(\mathrm{T}_{\tr})$ be the type of some pair $(u,x)$ where $u$ is a Haar unitary and $x$ has trace zero, such that $\Ent^{\cU}(\nu_j) = 0$.  Then there exist $2$-variable nonnegative definable predicates $\phi_1$, \dots, $\phi_k$ with $\nu_j[\phi_j] = 0$ for each $j$ and there exists a modulus of continuity $w: [0,\infty) \to [0,\infty)$, such that $\mu_{\Haar}$ annihilates the formula $\psi$ given by
		\begin{equation} \label{eq: implication formula psi}
			\psi^{\cQ}(u_1,u_2,\dots) := \sup_{x_1, \dots, x_k \in B_1^{\cQ}} \left( |\tr^{\cQ}(x_1 \dots x_k)| \dot{-} w(\max_j \phi_j^{\cQ}(u_{i_j},x_j)) \right),
		\end{equation}
		where $a \dot{-} b = \max(a-b,0)$.
	\end{lem}
	
	\begin{proof}
		From Lemma \ref{lem: type asymptotic freeness}, we know that if $(u_1,u_2,\dots)$ realizes the type $\mu_{\Haar}$ and if $\tp^{\cQ}(u_{i_j},x_j) = \nu_j$ for $j = 1$, \dots, $k$, then $\tr^{\cQ}(x_1 \dots x_k) = 0$.
		
		By Lemma \ref{lem: isolate type}, fix a $2$-variable definable predicate $\phi_j \geq 0$ such that for $u, x \in B_1^{\cQ}$, we have $\phi_j^{\cQ}(u,x) = 0$ if and only if $\tp^{\cQ}(u,x) = \nu_j$.  We claim that for every $\epsilon > 0$, there exists a $\delta > 0$ such that if $\max_j \phi_j^{\cQ}(u_{i_j},x_j) < \delta$, then $|\tr^{\cQ}(x_1 \dots x_k)| < \epsilon$.  Suppose for contradiction that this fails.  Then there exists some $\epsilon > 0$ such that for all $\delta > 0$, there exist $x_1, \dots, x_k$ with $\phi_j^{\cQ}(u_{i_j},x_j) < \delta$ for each $j$ but $|\tr^{\cQ}(x_1 \dots x_k)| \geq \epsilon$.  Consider the definable predicates $\{ \phi_j(u_{i_j},x_j) \}_{j=1}^k \cup \{ \epsilon \dot{-} |\tr(x_1 \dots x_k)| \}$ in variables $(x_1,\dots,x_k)$ and constants $(u_1,u_2,\dots)$.  Then this set of formulas is approximately satisfiable, and so by countable saturation of $\cQ$ (Lemma \ref{lem: saturation}), it is satisfiable.  That is, there exist $(x_1,\dots,x_k)$ with $\phi_j(u_{i_j},x_j) = 0$ and $|\tr^{\cQ}(x_1 \dots x_k)| \geq \epsilon$.  This contradicts the conclusion of Lemma \ref{lem: type asymptotic freeness} that we stated at the beginning of the proof.
		
		Knowing that for every $\epsilon > 0$, there exists a $\delta > 0$ such that if $\max_j \phi_j^{\cQ}(u_{i_j},x_j) < \delta$, then $|\tr^{\cQ}(x_1 \dots x_k)| < \epsilon$, one can construct a continuous increasing function $w: [0,\infty) \to [0,\infty)$ with $w(0) = 0$ such that $|\tr^{\cQ}(x_1 \dots x_k)| \leq w(\max_j \phi_j^{\cQ}(u_{i_j},x_j))$ for all $x_1,\dots,x_k \in B_1^{\cQ}$.
	\end{proof}
	
	\begin{rem} \label{rem: need Hausdorff}
		Lemma \ref{lem: type implication} is where we rely the fact on working with the full types rather than the existential types.  Since the space of existential types is not Hausdorff, one cannot use continuous functions to separate points in that space.  Hence, it is necessary to reason about a larger space which is compact and Hausdorff.
	\end{rem}
	
	\begin{thm} \label{thm: vanishing entropy freeness}
		Let $(u_1,u_2,\dots)$ be a tuple in $\cQ$ with type equal to $\mu_{\Haar}$ of Lemma \ref{lem: Haar type}.  For each $i \in \bN$, let $\cA_i \ni u_i$ with $\Ent^{\cU}(\cA_i:\cM) = 0$.  Then $\cA_1$, $\cA_2$, \dots are freely independent.
	\end{thm}
	
	\begin{proof}
		Let $i_1 \neq i_2 \neq \dots \neq i_k$.  Let $x_j \in \cA_{i_j}$ with $\tr^{\cQ}(x_j) = 0$, and we will show that $\tr^{\cQ}(x_1 \dots x_k) = 0$.  Let $\nu_j$ be the type of $(u_j,x_j)$.  Note that by \eqref{eq: monotonicity for type entropy},
		\[
		\Ent^{\cU}(\nu_j) = \Ent^{\cU}(\mathrm{W}^*(u_j,x_j):\cM) \leq \Ent^{\cU}(\cA_j:\cM) = 0.
		\]
		Therefore, by Lemma \ref{lem: type implication}, there exists some modulus of continuity $w$ such that $\mu_{\Haar}$ vanishes on the formula $\psi$ of \eqref{eq: implication formula psi}.  In particular, we have
		\[
		|\tr^{\cQ}(x_1 \dots x_k)| \dot{-} w(\max_j \phi_j^{\cQ}(u_{i_j},x_j)) = 0.
		\]
		Since $\phi_j^{\cQ}(u_{i_j},x_j) = \nu_j[\phi_j] = 0$, we obtain that $\tr^{\cQ}(x_1 \dots x_k) = 0$ as desired.  Since this holds for every alternating string $i_1 \neq \dots \neq i_k$ and all trace-zero $x_j \in \cA_{i_j}$, we have that $\cA_1$, $\cA_2$, \dots are freely independent as desired.
	\end{proof}
	
	\begin{proof}[Proof of Theorem \ref{thm: freeness of Pinsker}]
		Fix an outcome $\omega$ in the almost sure event where the type $(U_1^{(n)},U_2^{(n)},\dots)$ converges to $\mu_{\Haar}$, or in other words the type of $(u_1(\omega), u_2(\omega),\dots)$ is $\mu_{\Haar}$.  Let $\cP_j$ be the Pinsker algebra containing $u_j$.  Since $\Ent^{\cU}(\cP_j:\cQ) \leq h(\cP_j:\cQ) = 0$ by \eqref{eq: existential entropy bound}, the $\cP_j$'s are freely independent by Theorem \ref{thm: vanishing entropy freeness}.
	\end{proof}
	
	\begin{rem}[On the theory of matrix ultraproducts] \label{rem: theory}
		Fix $2$-variable types $\nu_j$ with $\Ent^{\cU}(\nu_j) = 0$, and let $\psi$ be as in Lemma \ref{lem: type implication}.  Note that $\psi$ only depends on the finitely many of $u_i$'s, and since the unitaries form a definable set, the following is a definable predicate with no free variables:
		\[
		\gamma = \inf_{u_1, u_2, \dots \text{ unitary}} \psi(u_1,u_2,\dots).
		\]
		Then the matrix ultraproduct $\cQ$ satisfies the sentence $\gamma^{\cQ} = 0$ since $\psi^{\cQ}$ vanishes on the unitaries with type $\mu_{\Haar}$.  By similar reasoning as in Lemma \ref{lem: type implication}, the free independence of the commutants of the $u_j$'s from Theorem \ref{thm: commutant} leads to the following statement:  Fix a word $i_1 \neq i_2 \neq \dots \neq i_k$. There is some modulus of continuity $w$ such that
		\[
		\inf_{u_1, u_2, \dots \in \cQ \text{ unitary}} \sup_{x_1, \dots, x_k \in B_1^{\cQ}} \left( |\tr^{\cQ}(x_1 \dots x_k) \dot{-} w(\max_j \norm{[u_{i_j},x_j]}_2) \right) = 0.
		\]
		Note that (for an appropriate choice of $w$) this sentence is also satisfied by $L(F_\infty)$ on account of \cite[Theorem B]{houdayer2023asymptotic} (or alternatively Theorem \ref{thm: free products} which we prove in the next section).  Hence, these freeness results give some new information about the $\exists \forall$ theory of $\cQ$ and that of $L(F_\infty)$, and in particular show agreement of the two theories on certain $\exists \forall$-sentences.
	\end{rem}

	
	\subsection{Free independence phenomena in ultraproducts of free products} \label{sec: free product phenomena}
	
	\begin{proof}[Proof of Theorem \ref{thm: free products}]
		First, to prove claims (1), (2), (3), it suffices to consider the case where $I$ is countable.  Indeed, for (1), we have to prove that for every $i_1 \neq i_2 \neq \dots \neq i_k$ if $x_j \in \cA_{i_j}$ with trace zero, then $\tr^{\cM^{\cV}}(x_1 \dots x_k) = 0$.  Let $I_0$ be a countable subset of $I$ containing the indices $i_1$, \dots, $i_k$.  Then view $\cM$ as a free product of countably many algebras $\cM_0 * (*_{i \in I_0} \cM_i)$, where $\cM_0 = *_{i \in I \setminus I_0} \cM_i$.  By applying the conclusions of the theorem for this free product decomposition, we obtain $\tr^{\cM^{\cV}}(x_1 \dots x_k) = 0$.  Hence, in the remainder of the proof assume without loss of generality that $I = \bN$.
		
		Now because $\cM_i$ is assumed to be diffuse and Connes embeddable, there exists an ultrafilter $\cV'$ on some infinite index set\footnote{Since $\cM_i$ is not necessarily separable, the index set for $\cV'$ may be uncountable.} such that $\cM_i$ embeds into $\cR^{\cV'}$ for each $i \in I$, where $\cR$ is the unique separable hyperfinite II$_1$ factor.  Since we assumed that $\cA_i \cap \cM_i$ is diffuse, let $a_i$ be a Haar unitary in $\cM_i$.  We may assume without loss of generality that inclusion $\cM_i \to \cR^{\cV'}$ sends $a_i$ into the diagonal copy of $\cR$ in $\cR^{\cV'}$; this is because all Haar unitaries in $\cR^{\cV'}$ are conjugate to each other.
		
		Let $\cS = *_{i \in I} \cR \cong L(F_\infty)$.  For convenience of notation, we will denote by $\cR_i$ the $i$th copy of $\cR$ in $\cS$.  The inclusions $\cR_i \to \cS$ yields an inclusion $\cR_i^{\cV'}$ into $\cS^{\cV'}$, and the $\cR_i^{\cV'}$'s are freely independent in $\cS^{\cV'}$ by a straightforward limiting argument.  Hence, we have mappings
		\[
		*_{i \in I} \cM_i \to *_{i \in I} \cR_i^{\cV'} \to (*_{i \in I} \cR_i)^{\cV'}.
		\]
		Thus also we have inclusions
		\[
		\cM^{\cV} = (*_{i \in I} \cM_i)^{\cV} \to (*_{i \in I} \cR_i^{\cV'})^{\cV} \to ((*_{i \in I} \cR_i)^{\cV'})^{\cV} \cong (*_{i \in I} \cR_i)^{\cW} = \cS^{\cW},
		\]
		where $\cW$ is a certain ultrafilter on the product of the index sets for $\cV'$ and $\cV$.  We have also arranged that the Haar unitary $a_i$ in $\cM_i \cap \cA_i \subseteq (*_{i \in I} \cM_i)^{\cV}$ is contained in $\cR_i \subseteq \cR_i^{\cW} \subseteq \cS^{\cW}$.
		
		Note that $h(\cM_i: \cS^{\cW}) \leq h(\cR_i^{\cV'}: \cS^{\cW}) = 0$ since $\cR^{\cV'}$ has property Gamma.  Let $\cP_i$ be the Pinsker algebra of $\cM_i$ in $\cS^{\cW}$, that is, the unique maximal subalgebra with $h(\cP_i: \cS^{\cW}) = 0$.  Of course, we have that $\cR_i \subseteq \cP_i$ since $\cR_i \cap \cM_i$ is diffuse.
		
		We claim that the $\cP_i$'s are freely independent (the three claims in the theorem statement will follow from this, as we explain at the end of the proof).  It suffices to show free independence of any separable subalgebras $\cB_i$ inside $\cP_i$ such that $\cR_i \subseteq \cB_i$.  Fix a free ultrafilter $\cU$ on $\bN$ and let $\cQ = \prod_{n \to \cU} M_n(\bC)$.  Let $A^{(n)} = \diag(1,\zeta_n,\dots,\zeta_n^{n-1})$, where $\zeta_n = e^{2\pi i/n}$, and let $a = [A^{(n)}]_{n \in \bN} \in \cQ$.  Fix an embedding $\pi_i: \cR_i \to \cQ$.  Since all Haar unitaries in $\cQ$ are conjugate, assume without loss of generality that $\pi_i(a_i) = a$.  Let $(V_i^{(n)})_{i \in \bN}$ be independent $n \times n$ Haar random unitaries.  By Proposition \ref{prop: conjugation matrix model} and Lemma \ref{lem: Haar type}, we have that almost surely $\tp^{M_n(\bC)}(V_1^{(n)}A^{(n)}(V_1^{(n)})^*,V_2^{(n)} A^{(n)} (V_2^{(n)})^*,\dots)$ converges to the type $\mu_{\Haar}$ of Lemma \ref{lem: Haar type}.  Thus, fix an outcome $\omega$ where this occurs, let $v_i = [V_i^{(n)}(\omega)]_{n \in \bN} \in \cQ$, and let $u_i = v_i a v_i^*$, so that the type of $(u_1,u_2,\dots)$ is $\mu_{\Haar}$.  Note that $v_i \pi_i(\cR_i) v_i^*$ contains $u_i$ and is amenable; therefore, using \eqref{eq: existential entropy bound} and \cite[\S 2.3, Property 4]{PropTS1B}.
		\[
		0 \leq \Ent^{\cU}(v_i \pi_i(\cR_i) v_i^*:\cQ) \leq h^{\cU}(v_i \pi_i(\cR_i)) v_i^*:\cQ) \leq h(v_i \pi_i(\cR_i)) v_i^*:\cQ) \leq 0.
		\]
		Thus, by Theorem \ref{thm: vanishing entropy freeness}, the algebras $v_i \pi_i(\cR_i) v_i^*$ are freely independent.  In particular, there is an embedding $\pi: \cS = *_{i \in I} \cR_i \to \cQ$ such that $\pi|_{\cR_i} = \operatorname{ad}_{v_j} \circ \pi_i$.
		
		By the downward L\"owenheim-Skolem theorem (Proposition \ref{prop: DLS}), choose some separable elementary substructure $\widehat{\cS}$ of $\cS^{\cW}$ containing all the $\cB_i$'s.  We claim that $\pi$ extends to an embedding $\widehat{\pi}: \widehat{\cS} \to \cQ$.  In more detail, let $\mathbf{x}$ and $\mathbf{y}$ be countable tuple generating $\cS$ and $\widehat{\cS}$ respectively, and assume they are elements in the unit ball.  Since $\cS \subseteq \widehat{\cS} \subseteq \cS^{\cW}$, then for every $\epsilon > 0$ and for any finitely many polynomials $p_1$, \dots, $p_\ell$ of countably many variables, there exists some tuple $\mathbf{z}$ from the unit ball in $\cS$ such that
		\[
		|\tr^{\cS}(p_j(\mathbf{x},\mathbf{z})) - \tr^{\cS^{\cW}}(p_j(\mathbf{x},\mathbf{y}))| < \epsilon.
		\]
		In particular, there exists $\mathbf{w}$ in $\cQ$ (namely $\mathbf{w} = \pi(\mathbf{z})$) such that
		\[
		|\tr^{\cQ}(p_j(\pi(\mathbf{x}),\mathbf{w})) - \tr^{\cS^{\cW}}(p_j(\mathbf{x},\mathbf{y}))| < \epsilon.
		\]
		Hence, by countable saturation of $\cQ$ (Lemma \ref{lem: saturation}), there exists $\mathbf{w}$ in $\cQ$ such that
		\[
		\tr^{\cQ}(p(\pi(\mathbf{x}),\mathbf{w})) = \tr^{\cM^{\cV}}(p(\mathbf{x},\mathbf{y}))
		\]
		for all non-commutative polynomials $p$.  Hence, there is a trace-preserving embedding $\widehat{\pi}: \widehat{\cS} \to \cQ$ such that $\widehat{\pi}(\mathbf{x}) = \pi(\mathbf{x})$ and $\widehat{\pi}(\mathbf{y}) = \mathbf{w}$.
		
		Now
		\[
		h(\widehat{\pi}(\cB_i): \cQ) \leq h(\widehat{\pi}(\cB_i): \widehat{\pi}(\widehat{\cS})) = h(\cB_i:\widehat{\cS}) = h(\cB_i:\cS^{\cW}) \leq h(\cP_i:\cS^{\cW}) = 0;
		\]
		here the equality $h(\cB_i:\widehat{\cS}) = h(\cB_i:\cS^{\cW})$ follows because $h(\cB_i:\widehat{\cS})$ only depends on the existential type of generators of $\cB_i$ in $\widehat{\cS}$, and this is the same as its existential type in $\cS^{\cW}$ since $\widehat{\cS}$ is an elementary substructure.  Since $\widehat{\pi}(\cB_i)$ by construction contains $\ad_{v_i} \circ \pi_i(\cR_i)$ and hence contains $v_iav_i^* = u_i$, we can apply Theorem \ref{thm: vanishing entropy freeness} to obtain that the algebras $\pi_i(\cB_i)$ are freely independent in $\cQ$.  This means also that the $\cB_i$'s are freely independent in $\cS^{\cW}$.
		
		Therefore, we have shown that the Pinsker algebras $\cP_i$ are freely independent.  The claims of the theorem now follow quickly from this, together with the properties of $h$:
		\begin{enumerate}
			\item If $\cA_i \subseteq \cM^{\cV} \subseteq \cS^{\cW}$ and $\cA_i \cap \cM_i$ is diffuse and $h(\cA_i: \cM^{\cV}) = 0$, then we obtain using \cite[\S 2.3, Property 3]{PropTS1B} that $h(\cA_i: \cS^{\cW}) \leq h(\cA_i: \cM^{\cV}) = 0$, and therefore $\cA_i$ is contained inside the Pinsker algebra $\cP_i$.  Thus, the $\cA_i$'s are freely independent.
			\item Similar to point (1), it suffices to note that $\cC_i$ is contained in the Pinsker algebra $\cP_i$, which follows from \cite[Fact 2.9]{patchellelayavalli2023sequential} as noted in the introduction at Corollary \ref{cor: freeness of SCorbit}.
			\item Note that the von Neumann $\cN_i$ generated by the wq-normalizer of $\cM_i$ in $\cM$ is contained in the von Neumann algebra $\tilde{\cN}_i$ generated by wq-normalizer of $\cM_i$ in $\cS^{\cW}$.  By \cite[\S 2.3, Properties 3 and 9]{PropTS1B}, we have $h(\tilde{\cN}_i: \cS^{\cW}) = h(\cM_i: \cS^{\cW}) \leq h(\cR_i^{\cV'}: \cS^{\cW}) = 0$.  Since $\tilde{\cN}_i \cap \cM_i$ is diffuse, we have that $\tilde{\cN}_i$ is contained in the Pinsker algebra $\cP_i$.  Thus, the $\cN_i$'s are freely independent as desired.  \qedhere
		\end{enumerate}
	\end{proof}
	
	\section*{Declarations}
	
	{\bf Data availability statement:}
	This manuscript has no associated data.
	
	{\bf Funding and/or Conflicts of interests/Competing interests:}
	DJ acknowledges funding from the grant ``Logic and $\mathrm{C}^*$-algebras'' from National Sciences and Engingeering Research Council of Canada; the Independent Research Fund of Denmark, grant 1026-00371B; and the Horizon Europe Marie Sk{\l}odowska-Curie Action FREEINFOGEOM.  SKE was funded by the National Science Foundation (US) grant DMS 2350049.  We thank the American Institute of Mathematics and the NSF for funding DJ's visit to California in April 2024 where this paper first was developed.
	
	The authors have no relevant financial or non-financial interests to disclose.
	
	\bibliographystyle{plain}
	\bibliography{matrix-ultraproducts}

\begin{thebibliography}{10}

\bibitem{amrutam2025strictcomparisonreducedgroup}
Tattwamasi Amrutam, David Gao, Srivatsav~Kunnawalkam Elayavalli, and Gregory
  Patchell.
\newblock Strict comparison in reduced group $c^*$-algebras, 2025.

\bibitem{anantharaman-popa}
Claire Anantharaman and Sorin Popa.
\newblock An introduction to $\mathrm{II}_1$ factors.
\newblock {\em book in progress}, 2016.

\bibitem{AGZ2009}
Greg~W. Anderson, Alice Guionnet, and Ofer Zeitouni.
\newblock {\em An Introduction to Random Matrices}.
\newblock Cambridge Studies in Advanced Mathematics. Cambridge University
  Press, 2009.

\bibitem{PTkilled}
Serban Belinschi and Mireille Capitaine.
\newblock Strong convergence of tensor products of independent {G.U.E.}
  matrices.
\newblock Preprint, arXiv:2205.07695, 2022.

\bibitem{BAG1997}
Gerard {Ben Arous} and Alice Guionnet.
\newblock Large deviations for {W}igner's law and {V}oiculescu's
  non-commutative entropy.
\newblock {\em Probab Theory Relat Fields}, 108:517--542, 1997.

\bibitem{BYBHU2008}
Ita{\"i} {Ben Yaacov}, Alexander Berenstein, C.~Ward Henson, and Alexander
  Usvyatsov.
\newblock Model theory for metric structures.
\newblock In Z.~Chatzidakis et~al., editor, {\em Model Theory with Applications
  to Algebra and Analysis, Vol. II}, volume 350 of {\em London Mathematical
  Society Lecture Notes Series}, pages 315--427. Cambridge University Press,
  2008.

\bibitem{Bl06}
B.~Blackadar.
\newblock {\em Operator algebras}, volume 122 of {\em Encyclopaedia of
  Mathematical Sciences}.
\newblock Springer-Verlag, Berlin, 2006.
\newblock Theory of $C^*$-algebras and von Neumann algebras, Operator Algebras
  and Non-commutative Geometry, III.

\bibitem{bordenave2023norm}
Charles Bordenave and Benoit Collins.
\newblock Norm of matrix-valued polynomials in random unitaries and
  permutations, 2023.
\newblock Preprint, arXiv:2304.05714.

\bibitem{BO08}
Nathanial~P. Brown and Narutaka Ozawa.
\newblock {\em {$C^*$}-algebras and finite-dimensional approximations},
  volume~88 of {\em Graduate Studies in Mathematics}.
\newblock American Mathematical Society, Providence, RI, 2008.

\bibitem{CGVvH2024strong2}
Chi-Fang Chen, Jorge Garza-Vargas, and Ramon van Handel.
\newblock A new approach to strong convergence {II}. the classical ensembles.
\newblock Preprint, arXiv:2412.00593.

\bibitem{exoticCIKE}
Ionu{\c t} Chifan, Adrian Ioana, and Srivatsav Kunnawalkam~Elayavalli.
\newblock An exotic {II}$_1$ factor without property {G}amma.
\newblock {\em Geometric and Functional Analysis}, 33:1243--1265, 2023.

\bibitem{dlSM2024}
Mikael de~la Salle and Michael Magee.
\newblock Strong asymptotic freeness of {H}aar unitaries in quasi-exponential
  dimensional representations.
\newblock preprint, arXiv:2409.03626, 2024.

\bibitem{dkep2022properly}
Changying Ding, Srivatsav Kunnawalkam~Elayavalli, and Jesse Peterson.
\newblock Properly proximal von {N}eumann algebras.
\newblock {\em Duke Math. Journal}, 172(15):2821--2894, 2023.

\bibitem{DP22}
Changying Ding and Jesse Peterson.
\newblock Biexact von {N}eumann algebras, 2023.
\newblock Preprint, arXiv:2309.10161.

\bibitem{FangGaoSmith}
J.~Fang, M.~Gao, and R.~Smith.
\newblock Weak asymptotic homomorphism property for inclusions of finite von
  {N}eumann algebras.
\newblock {\em Internat. J. Math.}, 22:991--1011, 2011.

\bibitem{FHS2013}
Ilijas Farah, Bradd Hart, and David Sherman.
\newblock Model theory of operator algebras {I}: stability.
\newblock {\em Bulletin of the London Mathematical Society}, 45(4):825--838,
  2013.

\bibitem{FHS2014a}
Ilijas Farah, Bradd Hart, and David Sherman.
\newblock Model theory of operator algebras {II}: model theory.
\newblock {\em Israel Journal of Mathematics}, 201(1):477--505, 2014.

\bibitem{FHS2014b}
Ilijas Farah, Bradd Hart, and David Sherman.
\newblock Model theory of operator algebras {III}: elementary equivalence and
  {II$_1$} factors.
\newblock {\em Bulletin of the London Mathematical Society}, 46(3):609--628,
  2014.

\bibitem{FJP2023}
Ilijas Farah, David Jekel, and Jennifer Pi.
\newblock Quantum expanders and quantifier reduction for tracial von neumann
  algebras.
\newblock {\em J. Symbolic Logic}, 2025.
\newblock To appear.

\bibitem{GalatanPopa}
Alin Galatan and Sorin Popa.
\newblock Smooth bimodules and cohomology of $\mathrm{II}_1$ factors.
\newblock {\em J. Inst. Math. Jussieu}, 16(1):155--187, 2017.

\bibitem{gao2024internal}
David Gao, Srivatsav Kunnawalkam~Elayavalli, Gregory Patchell, and Hui Tan.
\newblock Internal sequential commutation and single generation.
\newblock {\em International Mathematics Research Notices}, 2025(8):rnaf103, 04
  2025.

\bibitem{GePrime}
Liming Ge.
\newblock Applications of free entropy to finite von {N}eumann algebras. {II}.
\newblock {\em Ann. of Math. (2)}, 147(1):143--157, 1998.

\bibitem{Goldbring2023spectralgap}
Isaac Goldbring.
\newblock Spectral gap and definability.
\newblock In J.~Iovino, editor, {\em Beyond First Order Model Theory, Volume
  II}, page~36. Chapman and Hall/CRC, 2023.

\bibitem{GH2023}
Isaac Goldbring and Bradd Hart.
\newblock A survey on the model theory of tracial von neumann algebras.
\newblock In Isaac Goldbring, editor, {\em Model Theory of Operator Algebras},
  pages 133--157. DeGruyter, Berlin, Boston, 2023.

\bibitem{GoldbringPi2023}
Isaac Goldbring and Jennifer Pi.
\newblock On the first-order free group factor alternative.
\newblock {\em J. Operator Theory}, 2025.
\newblock To appear.

\bibitem{GZ2000}
Alice Guionnet and Ofer Zeitouni.
\newblock {Concentration of the Spectral Measure for Large Matrices}.
\newblock {\em Electronic Communications in Probability}, 5(none):119--136,
  2000.

\bibitem{Hart2023}
Bradd Hart.
\newblock An introduction to continuous model theory.
\newblock In Isaac Goldbring, editor, {\em Model Theory of Operator Algebras},
  pages 83--131. DeGruyter, Berlin, Boston, 2023.

\bibitem{Hayes2018}
Ben Hayes.
\newblock 1-bounded entropy and regularity problems in von {N}eumann algebras.
\newblock {\em Int. Math. Res. Not. IMRN}, 2018(1):57--137, 2018.

\bibitem{HayesPT}
Ben Hayes.
\newblock A random matrix approach to the {P}eterson-{T}hom conjecture.
\newblock {\em Indiana Univ. Math. J.}, 71(3):1243--1297, 2022.

\bibitem{hayes2025selflessreducedfreeproduct}
Ben Hayes, Srivatsav~Kunnawalkam Elayavalli, and Leonel Robert.
\newblock Selfless reduced free product $c^*$-algebras, 2025.

\bibitem{PropTS1B}
Ben Hayes, David Jekel, and Srivatsav Kunnawalkam~Elayavalli.
\newblock Property {(T)} and strong 1-boundedness for von neumann algebras.
\newblock Preprint, arXiv:2107.03278, to apprear in J. Math. Inst. Jussieu.

\bibitem{hayes2023consequences}
Ben Hayes, David Jekel, and Srivatsav {Kunnawalkam Elayavalli}.
\newblock Consequences of the random matrix solution of the {P}eterson-{T}hom
  conjecture.
\newblock {\em Analysis and PDE}, 18(7):1805--1834, 2025.

\bibitem{freePinsker}
Ben Hayes, David Jekel, Brent Nelson, and Thomas Sinclair.
\newblock A random matrix approach to absorption in free products.
\newblock {\em Int. Math. Res. Not. IMRN}, 2021(3):1919--1979, 2021.

\bibitem{houdayer2023asymptotic}
Cyril Houdayer and Adrian Ioana.
\newblock Asymptotic freeness in tracial ultraproducts.
\newblock {\em Forum of Mathematics, Sigma}, 12:e88, 2024.

\bibitem{IPP08}
Adrian Ioana, Jesse Peterson, and Sorin Popa.
\newblock Amalgamated free products of weakly rigid factors and calculation of
  their symmetry groups.
\newblock {\em Acta Math.}, 200(1):85--153, 2008.

\bibitem{JekelCoveringEntropy}
David Jekel.
\newblock Covering entropy for types in tracial $\mathrm{W}^*$-algebras.
\newblock {\em Journal of Logic and Analysis}, 15(2):1--68, 2023.

\bibitem{JekelModelEntropy}
David Jekel.
\newblock Free probability and model theory of tracial $\mathrm{W}^*$-algebras.
\newblock In Isaac Goldbring, editor, {\em Model Theory of Operator Algebras},
  pages 215--267. DeGruyter, Berlin, Boston, 2023.

\bibitem{JekelDefinableClosure}
David Jekel.
\newblock Optimal transport for types and convex analysis for definable
  predicates in tracial $\mathrm{W}^*$-algebras.
\newblock {\em J. Funct. Anal.}, 287(9):110583, 2024.

\bibitem{Jung2007}
Kenley Jung.
\newblock Strongly $1$-bounded von {N}eumann algebras.
\newblock {\em Geom. Funct. Anal.}, 17(4):1180--1200, 2007.

\bibitem{patchellelayavalli2023sequential}
Srivatsav {Kunnawalkam Elayavalli} and Gregory Patchell.
\newblock Sequential commutation in tracial von neumann algebras.
\newblock {\em Journal of Functional Analysis}, 288(4):110719, 2025.

\bibitem{Ledoux2001}
Michel Ledoux.
\newblock {\em The concentration of measure phenomenon}, volume~89 of {\em
  Mathematical Surveys and Monographs}.
\newblock American Mathematical Society, Providence, RI, 2001.

\bibitem{louder2022strongly}
Larsen Louder and Michael Magee.
\newblock Strongly convergent unitary representations of limit groups.
\newblock {\em arXiv preprint arXiv:2210.08953}, 2022.

\bibitem{Meckes2019}
Elizabeth~S. Meckes.
\newblock {\em The Random Matrix Theory of the Classical Compact Groups}.
\newblock Cambridge Tracts in Mathematics. Cambridge University Press, 2019.

\bibitem{Meckes2013}
Elizabeth~S. Meckes and Mark~W. Meckes.
\newblock Spectral powers of random matrices.
\newblock {\em Electron. Comm. Probab.}, 18(78), 2013.

\bibitem{mei}
Tao Mei and \'{E}ric Ricard.
\newblock Free {H}ilbert transforms.
\newblock {\em Duke Math. J.}, 166(11):2153--2182, 2017.

\bibitem{MuvN43}
F.~J. Murray and J.~von Neumann.
\newblock On rings of operators. {IV}.
\newblock {\em Ann. of Math. (2)}, 44:716--808, 1943.

\bibitem{OzawaSolidActa}
Narutaka Ozawa.
\newblock Solid von {N}eumann algebras.
\newblock {\em Acta Math.}, 192(1):111--117, 2004.

\bibitem{OzPo10}
Narutaka Ozawa and Sorin Popa.
\newblock On a class of {II}{$_1$} factors with at most one {C}artan
  subalgebra, {II}.
\newblock {\em Amer. J. Math.}, 132(3):841--866, 2010.

\bibitem{Parraud2024tensor}
F{\'e}lix Parraud.
\newblock The spectrum of a tensor of random and deterministic matrices.
\newblock Preprint, arXiv:2410.04481, 2024.

\bibitem{Peterson2009}
Jesse Peterson.
\newblock A $1$-cohomology characterization of property (t) in von neumann
  algebras.
\newblock {\em Pacific Journal of Math.}, 243(1):181--199, 2009.

\bibitem{PetersonDeriva}
Jesse Peterson.
\newblock {$L^2$}-rigidity in von {N}eumann algebras.
\newblock {\em Invent. Math.}, 175(2):417--433, 2009.

\bibitem{PetersonThom}
Jesse Peterson and Andreas Thom.
\newblock Group cocycles and the ring of affiliated operators.
\newblock {\em Invent. Math.}, 185(3):561--592, 2011.

\bibitem{Pisier1}
Gilles Pisier.
\newblock Quantum expanders and geometry of operator spaces.
\newblock {\em J. Eur. Math. Soc. (JEMS)}, 16(6):1183--1219, 2014.

\bibitem{PisierSubExp}
Gilles Pisier.
\newblock Random matrices and subexponential operator spaces.
\newblock {\em Israel J. Math.}, 203(1):223--273, 2014.

\bibitem{Pisier2}
Gilles Pisier.
\newblock Quantum expanders and growth of group representations.
\newblock {\em Ann. Fac. Sci. Toulouse Math. (6)}, 26(2):451--462, 2017.

\bibitem{PopaMaximalAmenable}
Sorin Popa.
\newblock Maximal injective subalgebras in factors associated with free groups.
\newblock {\em Adv. in Math.}, 50(1):27--48, 1983.

\bibitem{Popa1995independence}
Sorin Popa.
\newblock Free-independent sequences in type $\mathrm{II}_1$ factors and
  related problems.
\newblock In {\em Recent advances in operator algebras (Orl{\'e}ans, 1992)},
  volume 232 of {\em Ast{\'e}risque}, page 187–202, 1995.

\bibitem{Popasolidity}
Sorin Popa.
\newblock On {O}zawa's property for free group factors.
\newblock {\em Int. Math. Res. Not. IMRN}, 2007(11):Art. ID rnm036, 10, 2007.

\bibitem{Popa2014independence}
Sorin Popa.
\newblock Independence properties in subalgebras of ultraproduct ii1 factors.
\newblock {\em J. Funct. Anal.}, 266:5818--5846, 2014.

\bibitem{PopaVaesFree}
Sorin Popa and Stefaan Vaes.
\newblock Unique {C}artan decomposition for {$\rm II_1$} factors arising from
  arbitrary actions of free groups.
\newblock {\em Acta Math.}, 212(1):141--198, 2014.

\bibitem{robert2023selfless}
Leonel Robert.
\newblock Selfless {C*}-algebras.
\newblock {\em arXiv preprint arXiv:2309.14188}, 2023.

\bibitem{Sakai1971}
Sh{\^o}ichir{\^o} Sakai.
\newblock {\em $\mathrm{C}^*$-algebras and $\mathrm{W}^*$-algebras}, volume~60
  of {\em Ergebnisse der {M}athematik und ihrer {G}renzgebiete}.
\newblock Springer-Verlag, Berlin Heidelberg, 1971.

\bibitem{tarski}
Christopher Schafhauser and Srivatsav Kunnawalkam~Elayavalli.
\newblock Negative resolution to the ${C}^*$-algebraic {T}arski problem, 2025.

\bibitem{Szarek}
Stanis\l aw~J. Szarek.
\newblock Metric entropy of homogeneous spaces.
\newblock In {\em Quantum probability ({G}da\'{n}sk, 1997)}, volume~43 of {\em
  Banach Center Publ.}, pages 395--410. Polish Acad. Sci. Inst. Math., Warsaw,
  1998.

\bibitem{VoicAsyFree}
Dan Voiculescu.
\newblock Limit laws for random matrices and free products.
\newblock {\em Invent. Math.}, 104(1):201--220, 1991.

\bibitem{VoiculescuFreeEntropy2}
Dan-Virgil Voiculescu.
\newblock The analogues of entropy and of {F}isher's information measure in
  free probability theory {II}.
\newblock {\em Invent. Math.}, 118(3):411--440, 1994.

\bibitem{VoiculescuFreeEntropy3}
Dan-Virgil Voiculescu.
\newblock The analogues of entropy and of {F}isher's information measure in
  free probability theory. {III}. {T}he absence of {C}artan subalgebras.
\newblock {\em Geom. Funct. Anal.}, 6(1):172--199, 1996.

\bibitem{Voiculescu1998}
Dan-Virgil Voiculescu.
\newblock A strengthened asymptotic freeness result for random matrices with
  applications to free entropy.
\newblock {\em Internat. Math. Res. Not. IMRN}, 1998(1):41--63, 1998.

\bibitem{vonNeumann1942}
John von Neumann.
\newblock Approximative properties of matrices of high finite order.
\newblock {\em Portugaliae mathematica}, 3(1):1--62, 1942.

\bibitem{Zhu1993}
Kehe Zhu.
\newblock {\em An Introduction to Operator Algebras}.
\newblock Studies in Advanced Mathematics. CRC Press, Ann Arbor, 1993.

\end{thebibliography}
	
\end{document}